\documentclass[a4paper,10pt]{article}
\usepackage{amssymb}
\usepackage{amsmath}
\usepackage{amsfonts}
\usepackage{enumerate}
\usepackage{hyperref}
\usepackage{graphicx}
\usepackage{rotating}
\usepackage{mathtools}
\numberwithin{equation}{section}
\numberwithin{figure}{section}
\numberwithin{table}{section}
\usepackage[amsmath,thmmarks]{ntheorem}
\usepackage{xcolor}
\usepackage{cite}
\usepackage[top=2.5cm,left=2.5cm,right=2.5cm,bottom=2.5cm]{geometry}
\newtheorem{theorem}{Theorem}[section]
\newtheorem{lemma}[theorem]{Lemma}
\newtheorem{prop}[theorem]{Proposition}

\newtheorem{defi}[theorem]{Definition}
\theorembodyfont{\normalfont}
\newtheorem{remark}[theorem]{Remark}
\theoremstyle{nonumberbreak}
\theoremsymbol{\ensuremath{\Box}}
\newtheorem{proof}{Proof}
\DeclareMathAccent{\Circ}{\mathalpha}{operators}{"17}
\newcommand{\interior}[1]{\Circ{#1}}
\DeclareMathOperator{\Div}{div}
\DeclareMathOperator{\grad}{grad}

\newcommand{\e}{\mathrm{e}}
\newcommand{\eps}{\varepsilon}
\newcommand{\boldb}{{\boldsymbol{b}}}

\newcommand{\boldn}{{\boldsymbol{n}}}
\newcommand{\boldu}{{\boldsymbol{u}}}

\newcommand{\domain}[1]{\mathcal{D}(#1)}
\newcommand{\range}[1]{\mathcal{R}(#1)}
\title{A solution decomposition for a singularly perturbed fourth-order problem}
\author{Sebastian Franz, 
        Katharina H\"ohne, 
        Marcus Waurick}

\begin{document}
\maketitle

\begin{abstract}
   We consider a singularly perturbed fourth-order problem with
   third-order terms on the unit square.
   With a formal power series approach, we decompose the solution into 
   solutions of reduced (third-order) problems and various layer
   parts. The existence of unique solutions for the problem itself and
   for the reduced third-order problems is also addressed.
\end{abstract}
\section{Introduction}

In the present paper we are concerned with the singularly perturbed problem
\begin{equation}\label{eq:fourthorder}
   \begin{split}
      L\psi:=\varepsilon \Delta^2 \psi + (\boldsymbol{b} \cdot \nabla)\Delta \psi -c \Delta \psi=f \hspace{2em}&\mbox{ in }\Omega=(0,1)^2,\\
      \psi=0\hspace{2em}&\mbox{ on }\Gamma=\partial \Omega,\\
      \partial_\boldn \psi=0\hspace{2em}&\mbox{ on }\Gamma,
   \end{split}
\end{equation}
where $\boldsymbol{b}=(b_1,b_2)$ with $b_1,b_2>0$ and $c> 0$ are given, and
the perturbation parameter $\eps$ is supposed to be very small with $0<\eps\ll 1$.

The problem \eqref{eq:fourthorder} arises from different physical models. 
In particular, the equations \eqref{eq:fourthorder} can be formally derived from the Oseen equations, 
that is, from the streamfunction-vorticity formulation of the Oseen equations.
In this context the parameter $\eps$ is the reciprocal of the Reynolds number.
If the Reynolds number gets very large, the flow is said to be turbulent.
Although the Oseen equations are usually considered as a model for the moderate Reynolds-number regime,
we are interested in the high-Reynolds number case and see the Oseen equations
as linearisation of the non-linear Navier-Stokes equations.

Apart from the motivation in fluid dynamics fourth-order problems are frequently studied, when 
modelling plate-bending problems.
In contrast to our problem, this kind of problems is well understood
and numerical analysis can be found, see \cite{BR80,BrennerSung05,GuzmanLeykekhmanNeilan12,EngelGHLMT02,FrRW12}, just to name a few.
The main difference, however, is that the equations treated in the references cited do not contain third-order terms.
Thus, the corresponding reduced problem is elliptic which simplifies the asymptotic analysis.

Our method of choice for finding a proper solution decomposition into an interior part (arising from 
solutions of third-order problems) and layer parts is the method of asymptotic expansions. This approach can, 
for instance, be found in \cite{KSt05,KSt07,LS01,JF96}, where it is applied to second-order problems. 
Roughly rephrasing the rationale of asymptotic expansions, we consider a reduced problem, where formally $\eps$ is set to $0$
and certain boundary conditions are neglected, to construct solution parts that are independent
of $\eps$. The misfit of certain boundary conditions is then corrected by boundary
layer terms. As an ansatz for the boundary layer terms, we choose functions that exponentially decay away from the boundary.

More precisely, with the exponentially decaying functions
\[
%  E_1(x) = e^{-b_1\frac{(1-x)}{\varepsilon}}
%  \quad\mbox{and}\quad
%  E_2(y) = e^{-b_2\frac{(1-y)}{\varepsilon}},
 E_1(x) = \e^{-b_1\frac{x}{\varepsilon}}
 \quad\mbox{and}\quad
 E_2(y) = \e^{-b_2\frac{y}{\varepsilon}},
\]
our final result reads as follows: Assuming appropriate conditions on the right-hand side $f$, the solution $\psi$ of 
\eqref{eq:fourthorder} admits a decomposition
\[
      \psi(x,y)=S(x,y) + \eps H(x,y)   E_1(x)
                       + \eps I(x,y)   E_2(y)
                       + \eps^2 J(x,y) E_1(x)E_2(y),
\]
where the functions $S,\,H,\,I$ and $J$ are bounded independently of $\eps$.
Similar decompositions do also hold for the first derivatives.
In the one-dimensional case solution decompositions already exist in the literature, see, for instance, \cite{MR1323738}.
The decompositions given there contain so-called ``weak layers'': A property shared by the decomposition derived in this exposition.
Here, we call a function \emph{weak layer}, if it is exponentially
decaying away from the boundary, its $L^\infty$-norm vanishes for $\eps\to 0$ but the $L^\infty$-norm of its derivative does not.

The structure of this paper is as follows.
The existence of unique solutions to \eqref{eq:fourthorder} is shown and some of the solution's properties are given in Section~\ref{sec:exist}. 
In this section, we also put \eqref{eq:fourthorder} into an (abstract) functional analytic perspective. 
The methods are based on (more or less standard) perturbation theory for linear operators and a crucial regularity result 
established in \cite{BR80}. In order to establish a connection between the data of \eqref{eq:fourthorder} and its solution
we present a stability estimate in Section~\ref{sec:stab}, which relies on results of \cite{MR91}.
Finally, Section~\ref{sec:decomposition} contains the derivation of the decomposition mentioned above.

\subsection*{Notation}
Throughout this article, the domain $\Omega$ will be the unit square $(0,1)^2$.
Unless it is clear from the context, we will denote by $\|\cdot\|_{X}$ the norm of the Banach space $X$. 
Moreover, we will frequently neglect the reference to the underlying domain $D$ of certain Lebesgue or Sobolev spaces, 
that is, we will write, for instance, $\|\cdot\|_{L^2}$ instead of $\|\cdot\|_{L^2(D)}$.
%  we denote the standard $L_p$-norm.
% If there is only one index $s$, it is the usual Sobolev norm $\|\cdot\|_s$ or
% the semi norm $|\cdot|_s$ respectively of the Sobolev space $H^s$.
% The normal derivative $\frac{\partial u}{\partial n}$ is sometimes written as $D_n u$.
% This allows us to generalise the boundary conditions as $D_n^k u$ for the
% Dirichlet ($k=0$) and Neumann case ($k=1$).

\section{Existence and uniqueness of solutions}\label{sec:exist}

This section is devoted to the proof of existence and uniqueness of solutions of \eqref{eq:fourthorder}, for any $f\in L^2(\Omega)$. 
We approach the problem from an abstract point of view in Subsection \ref{ssec:per}. 
The findings are then applied to \eqref{eq:fourthorder} in Subsection \ref{ssec:4th}.

\subsection{Perturbation theorems for operators in Hilbert spaces}\label{ssec:per}

In this subsection we shall elaborate on some elements of perturbation theory of linear operators in Hilbert spaces. 
Large parts can be found in the standard reference \cite[Chapter III]{kato}. As some of the well-known results focus 
on selfadjoint/symmetric operators, some of the theorems might not be directly applicable. Hence, we provide full proofs of the 
respective results, though the general techniques employed are not new. 
The final aim of this section is to establish a proof of Theorem \ref{thm:sol_thy}. 
For stating and eventually proving the theorem, we need the following notion of relative boundedness.

\begin{defi} Let $H$ be a Hilbert space, $A:\domain{A}\subseteq H\to H$, $B\colon \domain{B}\subseteq H\to H$ be linear operators. 
$B$ is called \emph{$A$-bounded} with $A$-bound $\kappa\geq 0$, if $\domain{B}\supseteq \domain{A}$ and for all $\kappa'>\kappa$ 
there exists $C_{\kappa'}>0$ such that for all $\phi\in \domain{A}$ the inequality
\[
   \|B\phi\|\leq \kappa'\|A\phi\|+C_{\kappa'}\|\phi\|
\]
holds true. The operator $B$ is called \emph{infinitesimally $A$-bounded}, if $B$ is $A$-bounded with $A$-bound $0$. 
\end{defi}

For later use, for a linear operator $A\colon \domain{A}\subseteq H\to H$ in a Hilbert space $H$, we denote the space $\domain{A}$, the domain of $A$,
endowed with the graph scalar product of $A$ by $D_A$. Recall that $A$ is a closed operator if and only if $D_A$ is a Hilbert space.

\begin{remark}\label{rem:inf_bdd} Let, in this remark, $\Omega\coloneqq (0,1)^n\subseteq \mathbb{R}^n$ for some $n\in \mathbb{N}$.

(a) Let $j\in \{1,\ldots,n\}$. Consider the operator $\partial_j \colon H_j^1(\Omega)\subseteq L^2(\Omega)\to L^2(\Omega), f\mapsto \partial_j f$, 
where $H_j^1(\Omega)\coloneqq \{f\in L^2(\Omega); \partial_j f\in L^2(\Omega)\}$. 
Then the operator $\partial_j$ is infinitesimally $\partial_j^2$ bounded. 
This is a direct consequence of \cite[p. 191]{kato}.

(b) For any $k,m\in \mathbb{N}_{>0}$, $k>m$, and $\kappa>0$ there exists $C_{\kappa}>0$ such that for all $\phi\in H^k(\Omega)$ the inequality
\begin{equation}\label{eq:rel_bdd_km}
  \|\phi\|_{H^m}\leq C_{\kappa}\|\phi\|_{L^2}+\kappa\|\phi\|_{H^k}
\end{equation}
holds true. In order to prove the claimed inequality \eqref{eq:rel_bdd_km}, we proceed by induction: 
The case $k\geq 2$ and $m=1$ is a consequence of part (a) and of the Lipschitz continuity of the embedding $H^k\hookrightarrow H^\ell$ 
for every $\ell\leq k$ with Lipschitz constant $1$. 
Next, assume the inequality \eqref{eq:rel_bdd_km} is valid for some $m\in \mathbb{N}_{>0}$ and all $k>m$. 
Let $\kappa>0$ and $k>m+1$. Employing the induction hypothesis, we find $C_\kappa>0$ such that for all $\phi\in H^{m+1}(\Omega)$ 
\begin{align*}
  \|\phi\|_{H^{m+1}} & \leq \|\phi\|_{L^2}+\sum_{j=1}^n \|\partial_j\phi\|_{H^m}\\ 
                     & \leq \|\phi\|_{L^2}+\sum_{j=1}^n \left(C_{\kappa}\|\partial_j\phi\|_{L^2}+\kappa\|\partial_j \phi\|_{H^{k-1}}\right).
\end{align*}
By part (a) and the fact that $k\geq 2$, for $\kappa'>0$ there exists $C_{\kappa'}>0$ such that we may estimate further
\begin{align*}
    \|\phi\|_{H^{m+1}} 
          & \leq \|\phi\|_{L^2(\Omega)}+\sum_{j=1}^n \left(C_{\kappa}C_{\kappa'}\|\phi\|_{L^2}+C_{\kappa}\kappa'\|\phi\|_{H^k(\Omega)}+\kappa\|\partial_j \phi\|_{H^{k-1}}\right)\\ 
          & \leq (1+n C_{\kappa}C_{\kappa'}) \|\phi\|_{L^2} +(nC_{\kappa}\kappa'+\kappa)\| \phi\|_{H^{k}}.
\end{align*}
The latter inequality yields the proof of the inductive step.

(c) Let $k,m\in \mathbb{N}_{>0}$, $k>m$. Let $A\colon \domain{A}\subseteq L^2(\Omega)\to L^2(\Omega)$ be closed with $\domain{A}\subseteq H^k(\Omega)$, 
$B\colon H^m(\Omega)\to L^2(\Omega)$ linear and bounded. Then $B$ considered as an operator in $L^2(\Omega)$ is infinitesimally $A$-bounded. 
Indeed, the canonical embedding $\iota\colon D_A \hookrightarrow H^k(\Omega)$ is well-defined. 
Moreover, as the mapping $D_A\hookrightarrow L^2(\Omega)$ is continuous, the operator $\iota$ is closed. 
Using the closedness of $A$, we infer that $D_A$ is a Hilbert space. Hence, $\iota$ is continuous by the closed graph theorem. 
In particular, there exists $C>0$ such that for all $\phi\in \domain{A}$ we have\footnote{In the applications to be discussed later on, 
the inequality will be guaranteed right away so that we do not really need to invoke the closed graph theorem.} 
\[
  \|\phi\|_{H^k}\leq C\|\phi\|_{D_A}
                \leq C\|A\phi\|_{L^2}+C\|\phi\|_{L^2}.
\]
Next, take $\phi\in \domain{A}$. Let $\kappa>0$ and let $C_{\kappa}>0$ as in \eqref{eq:rel_bdd_km}. 
Then, denoting by $\|B\|_{H^m\to L^2}$ the operator norm of $B$ as an operator mapping from $H^m$ to $L^2$, we compute for $\phi\in \domain{A}$
\begin{align*}
   \|B\phi\|_{L^2} & \leq \|B\|_{H^m\to L^2}\|\phi\|_{H^m}\\ 
      & \leq \|B\|_{H^m\to L^2}C_{\kappa} \|\phi\|_{L^2}+\|B\|_{H^m\to L^2}\kappa \|\phi\|_{H^k}\\  
      & \leq  \|B\|_{H^m\to L^2}C_{\kappa} \|\phi\|_{L^2}+\|B\|_{H^m\to L^2}\kappa (C\|A\phi\|_{L^2}+C\|\phi\|_{L^2})\\  
      & = (\|B\|_{H^m\to L^2}C_{\kappa}+\|B\|_{H^m\to L^2}\kappa C) \|\phi\|_{L^2}+\kappa \|B\|_{H^m\to L^2} C\|A\phi\|_{L^2},
\end{align*}
which implies the assertion.
\end{remark}

\begin{lemma}\label{lem:rel_contr_closed_sum} 
   Let $H$ be a Hilbert space, $A\colon \domain{A}\subseteq H\to H$, $B\colon \domain{B}\subseteq H\to H$ be linear operators. 
   Assume that $A$ is closed and that $B$ is $A$-bounded with $A$-bound $\kappa<1$. Then $A+B$ is closed with $\domain{A+B}=\domain{A}$
\end{lemma}
\begin{proof}
   Note that, by the $A$-boundedness of $B$, $\domain{B}\supseteq \domain{A}$. Hence, the natural domain of $A+B$ coincides with the one of $A$. 
   Thus, only the closedness needs to be shown. For this, let $C>0$ be such that $\|Bx\|\leq C\|x\|+\kappa'\|Ax\|$ for some $\kappa'<1$. 
   Then, we compute
   \[
      \|Ax\|\leq \|(A+B)x\|+\|Bx\|
            \leq \|(A+B)x\|+C\|x\|+\kappa'\|Ax\| \quad(x\in \domain{A})
   \]
   Hence, $\|x\|+(1-\kappa')\|Ax\| \leq \|(A+B)x\|+(C+1)\|x\|$ $(x\in \domain{A})$. On the other hand, we realize
   \[
     \|(A+B)x\|+\|x\|\leq (1+\kappa)\|Ax\|+(C+1)\|x\|\quad(x\in \domain{A}).
   \]
   Therefore, the norms $\|\cdot\|+\|A\cdot\|$ and $\|\cdot\|+\|(A+B)\cdot\|$ are equivalent, proving the assertion.
\end{proof}

We recall that if $A$ is a selfadjoint operator in a Hilbert space $H$, that is, we have $A=A^*$, then the operator $A$ is necessarily closed and densely defined. 

For properly computing the adjoint of the sum $A+B$, we need a condition on $B^*$:

\begin{lemma}\label{lem:form_for_adjoint}  Let $H$ be Hilbert space, $A\colon \domain{A}\subseteq H\to H$, $B\colon\domain{A}\subseteq H\to H$ linear operators. Assume that $A$ is selfadjoint, $\domain{B}\supseteq \domain{A}$ and that $B^*$ is $A$-bounded with $A$-bound $<1$. Then 
\[
   (A+B)^*=A+B^*
\]
\end{lemma}

We need the following prerequisit.

\begin{lemma}\label{lem:B_res_small}
 Let $H$ be a Hilbert space, $A\colon \domain{A}\subseteq H\to H$, $B\colon \domain{B}\subseteq H\to H$ be linear operators. Assume that $A$ is selfadjoint and that $B$ is $A$-bounded with $A$-bound $<1$. Then there exists $z\in \rho(A)$ such that 
 \[
    \|B(A-z)^{-1}\|<1.
 \] 
\end{lemma}
\begin{proof}
 By hypothesis, there exists $0\leq \kappa<1$ and $C_\kappa\geq 0$ such that for all $x\in \domain{A}$, we have
 \[
    \|Bx\|\leq \kappa \|Ax\|+ C_\kappa \|x\|.
 \]
 Next, as for any $r\in \mathbb{R}\setminus\{0\}$ by the self-adjointness of $A$, we infer
 \[
   \|(A-ir)^{-1}\|\leq \frac{1}{|r|}.
 \]
 Thus, we get for all $x\in H$ and $r\in \mathbb{R}\setminus\{0\}$,
 \begin{align*}
    \|B(A-ir)^{-1}x\|^2 & \leq \left(\kappa \|A(A-ir)^{-1}x\|+ C_\kappa \|x\|\right)^2\\
    \\ & \leq \left(\kappa^2 \|A(A-ir)^{-1}x\|^2+2\kappa C_\kappa\|A(A-ir)^{-1}x\|\|(A-ir)^{-1}x\|+ C_\kappa^2 \|(A-ir)^{-1}x\|^2\right)
    \\ & \leq \kappa^2\|x\|^2+ 2\kappa C_\kappa\frac{1}{|r|} \|x\|^2+C_\kappa^2\frac{1}{r^2}\|x\|^2
    \\ & \leq \left(\kappa^2+2\kappa C_\kappa\frac{1}{|r|}+C_\kappa^2\frac{1}{r^2}\right) \|x\|^2,
 \end{align*}
where we used the selfadjointness of $A$ to estimate $\|A(A-ir)^{-1}\|\leq 1$.
\end{proof}

Next, we prove Lemma \ref{lem:form_for_adjoint}. The proof of which has been kindly communicated to the authors by Rainer Picard.

\begin{proof}[Lemma \ref{lem:form_for_adjoint}]
  By Lemma \ref{lem:B_res_small}, we find $z\in \rho(A)$ such that $\|B^*(A-z)^{-1}\|<1$. Next, we observe that 
  \[
     \overline{(A-z^*)^{-1}B}=(B^*(A-z)^{-1})^*.
  \]
 In fact, the equality being clear on the domain of $B$, so the equality is plain since $B$ is densely defined and the operator on the right-hand side is continuous. In particular, we infer that $\|\overline{(A+\mu)^{-1}B}\|<1$ for $\mu\coloneqq -z^*$. 
  
  For proving the claim of the lemma, we take $x\in \domain{(A+B+\mu)^*}$. Then for all $y\in \domain{A+B+\mu}$, the equality 
  \[
     \langle (A+B+\mu)y,x\rangle = \langle y,(A+B+\mu)^*x\rangle
  \]
  holds true. For $y\in \domain{A+B+\mu}$ putting $u_y\coloneqq (1+\overline{(A+\mu)^{-1}B})y$, we compute
  \[
    (A+B+\mu)y = (A+\mu)y+By = (A+\mu)(1+(A+\mu)^{-1}B)y= (A+\mu)u_y.
  \]
  Hence, 
  \[
   \langle (A+\mu)u_y,x\rangle = \langle (1+\overline{(A+\mu)^{-1}B)}^{-1} u_y, (A+B+\mu)^*x \rangle,
  \]
   or, expressed differently, for all $w\in \range{(1+(A+\mu)^{-1}B)|_{\domain{A}}}$ we infer
  \[
   \langle (A+\mu)w,x\rangle = \langle (1+\overline{(A+\mu)^{-1}B)}^{-1} w, (A+B+\mu)^*x \rangle.
  \]
  Next, observe that $\range{(1+(A+\mu)^{-1}B)|_{\domain{A}}}$ is dense in $H$, since $\domain{A}$ is dense in $H$ and $(1+\overline{(A+\mu)^{-1}B})$ is an isomorphism. Therefore, the continuous extension of the functional 
  \[
     \range{(1+(A+\mu)^{-1}B)|_{\domain{A}}}\ni w\mapsto \langle (1+\overline{(A+\mu)^{-1}B)}^{-1} w, (A+B+\mu)^*x \rangle
  \]
  defines an element of $\domain{A^*}$. Hence, $x\in \domain{(A+\mu)^*}=\domain{A}$ and 
  \[
     (A+\mu)^*x=(1+(B^*((A+\mu)^*)^{-1}))^{-1} (A+B+\mu)^*x \rangle,
  \]
  or
  \[
     (A+B)^*x+\mu^*x=(A+B+\mu)^*x= (1+(B^*((A+\mu)^*)^{-1}))(A+\mu)^*x= (A+\mu^*+B^*)x,
  \]
  which yields the claim.  
\end{proof}

We finally come to the main result of this subsection. We recall that an operator $A\colon \domain{A}\subseteq H\to H$ 
is called \emph{non-negative}, if for all $\phi\in \domain{A}$, we have that $\langle A\phi,\phi\rangle\geq 0$.

\begin{theorem}\label{thm:sol_thy} 
   Let $H$ be a Hilbert space, $A\colon \domain{A}\subseteq H\to H$ be selfadjoint and non-negative. 
   Assume that $B\colon \domain{B}\subseteq H\to H$ is linear, $B$ and $B^*$ being $A$-bounded with $A$-bound $\kappa<1$. Assume there exists $c>0$ such that for all $\phi\in \domain{A}$ we have
   \[
      \Re\langle B\phi,\phi\rangle_H\geq c\langle\phi,\phi\rangle_H\text{ and }\Re\langle B^*\phi,\phi\rangle_H\geq c\langle\phi,\phi\rangle_H.
   \]
   Then the operator $A+B$ is continuously invertible with $\|(A+B)^{-1}\|\leq 1/c$. 
\end{theorem}
\begin{proof}
   By Lemma \ref{lem:rel_contr_closed_sum}, the operator $A+B$ is closed with $\domain{A+B}=\domain{A}$. 
   The same is true for the operator $A+B^*$.
   Next, for $\phi\in \domain{A}$ using the non-negativity of $A$, we estimate
   \begin{equation}\label{eq:inesol}
      \Re \langle (A+B)\phi,\phi\rangle \geq \Re\langle B\phi,\phi\rangle \geq c\langle\phi,\phi\rangle_H.
   \end{equation}
   The latter inequality \eqref{eq:inesol} implies that $A+B$ is invertible on its range $\range{A+B}\subseteq H$ 
   with Lipschitz constant $1/c$. We emphasize that the closedness of $A+B$ together with \eqref{eq:inesol} 
   also implies that $\range{A+B}\subseteq H$ is closed. From Lemma \ref{lem:form_for_adjoint}, we deduce that $(A+B)^*=A+B^*$. 
   Hence, from
   \[
      \Re \langle (A+B)^*\phi,\phi\rangle = \Re \langle (A+B^*)\phi,\phi\rangle \geq \Re\langle B^*\phi,\phi\rangle \geq c\langle\phi,\phi\rangle_H\quad(\phi\in \domain{(A+B)^*}=\domain{A})
   \]
   it follows that $(A+B)^*$ is one-to-one. The decomposition $H=\mathcal{N}((A+B)^*)\oplus \overline{\range{A+B}}$ thus yields that $A+B$ is onto, as $\range{A+B}$ is closed.
\end{proof}

\subsection{The fourth-order problem}\label{ssec:4th}

In this section, we provide the well-posedness theorem for the fourth-order problem under consideration, 
see \eqref{eq:fourthorder}. For this, we start out with the problem with $\textbf{b}$ and $c$ formally set to zero. 
We introduce some operators from vector calculus in order to formulate the fourth-order problem in a proper 
functional analytic framework.
\begin{defi} 
   Let $\Omega\subseteq \mathbb{R}^n$ open. Then define
%    \begin{align*}
%        \grad_c \colon C_c^\infty(\Omega)\subseteq L^2 (\Omega)&\to L^2 (\Omega),\\
%                                                          \phi&\mapsto \left(\partial_j\phi\right)_{j\in\{1,\ldots,n\}}\\
%        \Div_c \colon C_c^\infty(\Omega)^n\subseteq L^2 (\Omega)^n&\to L^2 (\Omega),\\
%                                                           (\phi_j)_{j\in\{1,\ldots,n\}}&\mapsto \sum_{j=1}^n\partial_j\phi_j\\
%        \Delta_c\colon C_c^\infty(\Omega)\subseteq L^2 (\Omega)&\to L^2 (\Omega),\\
%                                                           \phi&\mapsto \sum_{j=1}^n\partial_j^2\phi.
%    \end{align*}
   \begin{alignat*}{3}
       \grad_c &\colon C_c^\infty(\Omega)\subseteq L^2 (\Omega)&\to& L^2 (\Omega),\quad&
                                                         \phi&\mapsto \left(\partial_j\phi\right)_{j\in\{1,\ldots,n\}}\\
       \Div_c  &\colon C_c^\infty(\Omega)^n\subseteq L^2 (\Omega)^n&\to& L^2 (\Omega),&
                                                          (\phi_j)_{j\in\{1,\ldots,n\}}&\mapsto \sum_{j=1}^n\partial_j\phi_j\\
       \Delta_c&\colon C_c^\infty(\Omega)\subseteq L^2 (\Omega)&\to& L^2 (\Omega),&
                                                          \phi&\mapsto \sum_{j=1}^n\partial_j^2\phi.
   \end{alignat*}
   Moreover, we set
   \[
    \grad\coloneqq -\Div_c^*,\quad \Div\coloneqq-\grad_c^*,\quad\Delta\coloneqq \Delta_c^*
   \]
   as well as
   \[
    \interior\grad\coloneqq \overline{\grad}_c,\quad \interior\Div\coloneqq\overline{\Div}_c,\quad\interior\Delta\coloneqq \overline{\Delta}_c.
   \] 
\end{defi}

In order to relate the operators just introduced to the equation under consideration, 
we investigate the domain of $\interior{\Delta}$ a bit further:

\begin{theorem}\label{thm:ass_of_bdy_cdt} 
   Let $\Omega=(0,1)^2$. Then $\interior{\Delta}=\interior{\Div}\,\interior{\grad}$, 
   in other words, $\phi\in \domain{\interior{\Delta}}$ satisfies both homogeneous 
   Dirichlet and Neumann boundary conditions. 
\end{theorem}
\begin{proof}
  From $\Delta|_{C_c^\infty(\Omega)}=\interior{\Div}\,\interior{\grad}|_{C_c^\infty(\Omega)}$ it follows that 
  $\interior{\Delta}=\overline{\interior{\Div}\,\interior{\grad}}$. Hence, it suffices to show that 
  $\interior{\Div}\,\interior{\grad}$ is closed. For this, let $(u_n)_n$ in $\domain{\interior{\Div}\,\interior{\grad}}$ 
  with $u_n\to f$ and $\interior{\Div}\,\interior{\grad} u_n\to g$ in $L^2(\Omega)$ as $n\to\infty$ for some 
  $f,g\in L^2(\Omega)$. Observe that for $n,m\in \mathbb{N}$, we have
   \[
      \|\interior{\grad}(u_n-u_m)\|^2=-\langle (u_n-u_m),\interior{\Div}\,\interior{\grad}(u_n-u_m)\rangle \to 0\quad (n,m\to\infty).
   \]
   Thus, $(u_n)_n$ is a Cauchy-sequence in $\domain{\interior{\grad}}$ endowed with the graph norm. 
   Thus, $f\in \domain{\interior{\grad}}$ and $\interior{\grad}f=\lim\limits_{n\to\infty} \interior\grad u_n$, 
   by the closedness of $\interior{\grad}$. Moreover, as both $(\interior{\grad} u_n)_n$ and 
   $(\interior{\Div}\,\interior{\grad}u_n)_n$ are convergent in $L^2(\Omega)^n$ and $L^2(\Omega)$, 
   respectively, we infer, by the closedness of $\interior\Div$ that $\interior{\grad}f\in \domain{\interior{\Div}}$ 
   and $\interior{\Div}\,\interior{\grad}f=g$, which proves the assertion.
\end{proof}

\begin{remark}[On the domain of $\interior{\Div}\,\interior{\grad}$] 
   (a) The boundary of $\Omega=(0,1)^2$ is Lipschitz. Hence, by a standard regularity result, we have
      \[
         \domain{\interior{\grad}}=H_0^1(\Omega)=\{u\in H^1(\Omega); u=0 \text{ on }\partial\Omega\}.
      \]
      Moreover, 
      \[
         \domain{\interior{\Div}}=\{ \boldu\in \domain{\Div};  \boldu\cdot \boldn=0 \text{ on }\partial\Omega\},
      \]
      where $\boldu \cdot \boldn$ is the normal component of $u$. Hence, the domain of $\interior{\Div}\,\interior{\grad}$ can be characterised by
      \[
        \domain{\interior{\Div}\,\interior{\grad}}=\{ u\in \domain{{\Div}\,{\grad}}; u=0, \partial_\boldn u=0 \text{ on }\partial\Omega\},
      \]
      where $\partial_\boldn u$ is the normal derivative of $u$. 

   (b) With the help of (a), the boundary conditions in equation \eqref{eq:fourthorder} may be interpreted equivalently in two ways. 
      The first way is that the boundary conditions may be understood in the sense of traces. Secondly, for a solution $\psi\in L^2(\Omega)$ 
      of \eqref{eq:fourthorder}, any summand should belong to $L^2(\Omega)$ and the equation (not including the boundary conditions) 
      should hold in a distributional sense. The boundary conditions are realised as the additional condition of $\psi$ belonging 
      to the domain of $\interior{\Delta}$. 
\end{remark}

\begin{remark}[On the regularity of $\domain{\interior{\Delta}}$-functions]\label{rem:reg_delta_int} 
   Note that Theorem \ref{thm:ass_of_bdy_cdt} in particular implies that $\interior{\Delta}\subseteq \Div\interior{\grad}\eqqcolon \Delta_{D}$, 
   where the index $D$ stands for homogeneous Dirichlet boundary conditions. So, $\interior{\Delta}$ is a restriction of the Dirichlet Laplacian. 
   It is known that, on convex domains, the Dirichlet Laplacian admits optimal regularity, that is, 
   $\domain{\Delta_D}=H^2(\Omega)\cap H_0^1(\Omega)$. 
   So, $\domain{\interior{\Delta}}=H_0^2(\Omega)\coloneqq \{ u\in H^2(\Omega); u=0, \partial_\boldn u =0 \text{ on }\partial\Omega\}$. 
\end{remark}

\begin{remark}\label{rem:ddstar_ci} 
   (a) It is a consequence of the Poincar\'e inequality that the operator $-\interior{\Delta}$ is strictly positive. 
      Indeed, we have for all $\phi\in \domain{\interior{\Delta}}$ that
      \[
         -\langle \interior{\Delta}\phi,\phi\rangle
            =-\langle \interior{\Div}\,\interior{\grad}\phi,\phi\rangle 
            =\langle \interior{\grad}\phi,\interior{\grad}\phi\rangle\geq c_P^2\|\phi\|_{L^2}^2
      \]
      for $c_P>0$ being the Poincar\'e constant on the square $\Omega=(0,1)^2$.  

   (b) The inequality in (a) shows that $\interior{\Delta}$ is one-to-one and has closed range. It is standard 
      (see, e.g.,~\cite[Corollary 2.5]{TW14}) that the range of $\interior{\Delta}^*=\Delta$ is closed itself. 
      As $\interior{\Delta}$ is one-to-one, the operator $\Delta$ is onto. Moreover, by \cite[Theorem 2.6]{TW14}, 
      the operator $|\Delta\interior{\Delta}|=\sqrt{\Delta\interior{\Delta}}$ is continuously invertible in $L^2(\Omega)$. 
      Hence, so is $|\Delta\interior{\Delta}|^2=\Delta\interior{\Delta}$.
\end{remark}

Next, we cite a crucial result for our approach, which asserts that -- similar to the Dirichlet Laplacian (cf.~Remark~\ref{rem:reg_delta_int}) -- the operator 
$\Delta\interior{\Delta}$ admits optimal regularity on $\Omega=(0,1)^2$, that is, $\domain{\Delta\interior{\Delta}}=H^4(\Omega)\cap H_0^2(\Omega)$.

\begin{theorem}[\!\!{{\cite[Theorem 2]{BR80}}}]\label{thm:BRT2} 
   The operator 
   \[
      \Delta\interior{\Delta} \colon \domain{\Delta\interior{\Delta}}\subseteq L^2(\Omega)\to L^2(\Omega)
   \]
   is continuously invertible and there exists $d>0$ such that for all $\phi\in \domain{\Delta\interior{\Delta}}$ we have the estimate
   \[
      \| \phi\|_{H^4(\Omega)}\leq d\|\Delta\interior{\Delta} \phi\|_{L^2(\Omega)}.
   \] 
\end{theorem}
\begin{proof}
    The continuous invertibility has been established in Remark \ref{rem:ddstar_ci} (b). 
    The regularity result is one of the statements in \cite[Theorem 2]{BR80}.
\end{proof}
For the proof of the continuous invertibility of the operator
\begin{equation}\label{eq:def_ofG}
    G\coloneqq \varepsilon\Delta\interior{\Delta}+(\textbf{b}\cdot\nabla)\interior{\Delta}-c\interior{\Delta}
\end{equation}
in $L^2(\Omega)$ for $\textbf{b}\in \mathbb{R}^2$ and $c>0$, we will employ the abstract results found in the previous section.

The next result shows that $\varepsilon\Delta\interior{\Delta}$ is in fact the leading term in the operator $G$. 
We briefly recall that any operator of the form $A^*A$ is selfadjoint and non-negative, 
where $A$ is a densely defined, closed linear operator in some Hilbert space.

\begin{lemma}\label{lem:Hinfbdd} 
   The operator  
%    \begin{align*}
%        T \colon \domain{\Delta\interior{\Delta}}\subseteq L^2(\Omega)& \to L^2(\Omega),\\  
%        u & \mapsto ((\textbf{b}\cdot\nabla)\interior{\Delta}-c\interior{\Delta})u
%    \end{align*}
   \[
       T \colon \domain{\Delta\interior{\Delta}}\subseteq L^2(\Omega) \to L^2(\Omega),\quad
       u  \mapsto ((\textbf{b}\cdot\nabla)\interior{\Delta}-c\interior{\Delta})u
   \]
   is infinitesimally $\Delta\interior{\Delta}$-bounded.
\end{lemma}
\begin{proof}
   By Theorem \ref{thm:BRT2}, $\domain{\Delta\interior{\Delta}}= H^4(\Omega)\cap H_0^2(\Omega)\subseteq H^4(\Omega)$. 
   Next, the operator 
   \[
      (\textbf{b}\cdot\nabla)\Delta-c\Delta \colon H^3(\Omega)\to L^2(\Omega)
   \]
   is continuous. In addition, note that the operator $\Delta\interior{\Delta}=\interior{\Delta}^*\interior{\Delta}$ 
   is selfadjoint, hence, closed. Thus, by Remark \ref{rem:inf_bdd}(c), considered in $L^2(\Omega)$, the operator 
   $(\textbf{b}\cdot\nabla)\Delta-c\Delta$ is infinitesimally $\Delta\interior{\Delta}$-bounded. 
   We conclude with the observation that on $\domain{\Delta\interior{\Delta}}$ the operator $T$ coincides 
   with $(\textbf{b}\cdot\nabla)\Delta-c\Delta$.
\end{proof}

\begin{lemma}\label{lem:Hisdos} 
   Let $T$ be given as in Lemma \ref{lem:Hinfbdd}. Then the inclusion $\domain{T}\subseteq \domain{T^*}$ holds.
\end{lemma}
\begin{proof}
   We will use again Theorem \ref{thm:BRT2}, that is $\domain{\Delta\interior{\Delta}}= H^4(\Omega)\cap H_0^2(\Omega)$. 
   So, for $\phi,\psi\in H^4(\Omega)\cap H_0^2(\Omega)$, we compute for $j\in \{1,2\}$ 
   \begin{align*}
     \langle \partial_j \interior{\Delta}\phi,\psi\rangle& = \langle \partial_j \Delta\phi,\psi\rangle \\
                                                            & = \langle \Delta\partial_j \phi,\psi\rangle \\
							    & = \langle \partial_j \phi,\interior{\Delta}\psi\rangle \\
 							    & = \langle \interior{\partial}_j \phi,\interior{\Delta}\psi\rangle \\
 							    & = -\langle  \phi,\partial_j\interior{\Delta}\psi\rangle,
   \end{align*}
   where we denoted by $\interior{\partial}_j$ the minimal closed restriction of the distributional 
   derivative operator $\partial_j$ with respect to the $j$'th coordinate in $L^2(\Omega)$ with $C_c^\infty(\Omega)$ as a core. 
   Moreover, we have
   \[
      -\langle\interior{\Delta}\phi,\psi\rangle =\langle\interior{\grad}\phi,\interior{\grad}\psi\rangle=-\langle\phi,\interior{\Delta}\psi\rangle,
   \]
   which establishes the assertion.
\end{proof}

\begin{remark}\label{rem:expr_for_Hstar} 
   Note that in the proof of Lemma \ref{lem:Hisdos}, we also showed that for $\phi\in \domain{\Delta\interior{\Delta}}$, we have
   \[
      T^*\phi = - (\textbf{b}\cdot\nabla)\interior{\Delta}\phi-c\interior{\Delta}\phi.
   \]
   Thus, from this equality and Lemma \ref{lem:Hinfbdd} it follows that $T^*$ is infinitesimally $\Delta\interior{\Delta}$-bounded, as well.
\end{remark}

\begin{lemma}\label{lem:Hstr_pos} 
   Let $T$ be given as in Lemma \ref{lem:Hinfbdd}. Then for all $\phi\in H^4(\Omega)\cap H_0^2(\Omega)$, 
   the estimates
   \[
      \Re\langle T\phi,\phi\rangle\geq c c_P^2\langle \phi,\phi\rangle\text{ and }\Re\langle T^*\phi,\phi\rangle\geq c c_P^2\langle \phi,\phi\rangle
   \]
   hold true. Here $c_P>0$ is the constant in Poincar\'e's inequality, see also Remark \ref{rem:ddstar_ci}.
\end{lemma}
\begin{proof} 
   Both the asserted inequalities are shown in the same way. Hence, we stick to the first inequality only. 
   We compute, using Remark \ref{rem:expr_for_Hstar}, for $\phi\in H^4(\Omega)\cap H_0^2(\Omega)$:
   \begin{align*}
      2\Re\langle T\phi,\phi\rangle & =  \Re\langle T\phi,\phi\rangle+ \Re(\langle T\phi,\phi\rangle)^*\\ 
         & =  \Re\langle T\phi,\phi\rangle+ \Re\langle \phi,T\phi\rangle\\ 
         & =  \Re\langle T\phi,\phi\rangle+ \Re\langle T^*\phi,\phi\rangle\\ 
         & = \Re\langle (\textbf{b}\cdot\nabla)\interior{\Delta}\phi-c\interior{\Delta}\phi,\phi\rangle 
            + \Re\langle - (\textbf{b}\cdot\nabla)\interior{\Delta}\phi-c\interior{\Delta}\phi,\phi\rangle\\ 
         & = -2\Re \langle c\interior{\Delta}\phi,\phi\rangle\\
         &\geq 2 c c_P^2 \langle\phi,\phi\rangle, 
   \end{align*} 
   where we also used Remark \ref{rem:ddstar_ci}(a).
\end{proof}

Finally, we are in the position to prove the well-posedness result for the operator $G$:

\begin{theorem}\label{thm:sol_thy_app}
   The operator $G$ as given in \eqref{eq:def_ofG} is continuously invertible with domain 
   $\domain{G}=H^4(\Omega)\cap H_0^2(\Omega)$. Moreover, $\|G^{-1}\|\leq 1/(c c_P^2)$, where 
   $c_P>0$ is the Poincar\'e constant given in Remark \ref{rem:ddstar_ci}.
\end{theorem}
\begin{proof}
   For the proof, we use Theorem \ref{thm:sol_thy} applied to $A=\varepsilon\Delta\interior{\Delta}$ 
   and $T=B$, where $T$ is given in Lemma \ref{lem:Hinfbdd}. We will now show that the assumptions of
   Theorem \ref{thm:sol_thy} are met. First of all note that the operator 
   $\varepsilon\Delta\interior{\Delta}=\varepsilon\interior{\Delta}^*\interior{\Delta}$ is selfadjoint 
   and non-negative. Next, the operators $T$ and $T^*$ are infinitesimally $\Delta\interior{\Delta}$-bounded by 
   Theorem \ref{lem:Hinfbdd} and Remark \ref{rem:expr_for_Hstar}. The needed inequalities are established in Lemma \ref{lem:Hstr_pos}.
\end{proof}
\subsection{A third-order problem}
We conclude Section \ref{sec:exist} with a brief summary of the results of \cite{MR2548434}, which are relevant for our analysis to be carried out later on.

\begin{defi}
   We call the problem
   \begin{equation}\label{eq:redfourthorder}
      \begin{split}
      (\boldsymbol{b} \cdot \nabla)\Delta \psi -c \Delta \psi=f \hspace{2em}&\mbox{ in }\Omega,\\
      \psi=0\hspace{2em}&\mbox{ on }\Gamma,\\
      \partial_\boldn \psi=0\hspace{2em}&\mbox{ on }\Gamma_-,
      \end{split}
   \end{equation}
   where $\Gamma_-:=\{x \in \Gamma: -\boldsymbol{b}\cdot \boldsymbol{n} <0\}$ is the inflow boundary, 
   the \emph{reduced problem} of \eqref{eq:fourthorder}.
\end{defi}
\begin{defi}
   A function $\psi$ from the class $C^{1,h}(\bar{\Omega})\cap C^3(\Omega)$ is a \emph{classical solution} 
   of the problem \eqref{eq:redfourthorder} if it satisfies the equation and its boundary conditions. 
   By $C^{1,h}(\Omega)$, we denote the space of functions which are in $C^1(\Omega)$ and whose derivatives satisfy the 
   H\"older condition for $0<h<1$.
\end{defi}
\begin{prop}
Let $f\in C^{1,h}(\bar{\Omega})$ and $c\geq0$. Then the problem \eqref{eq:redfourthorder} has a classical solution which is unique.
\end{prop}
\begin{proof}
For the proof, see \cite[Theorems 1 \& 2]{MR2548434}.
\end{proof}
In \cite{MR2548434}, Zikirov shows the unique solvability for a more general problem, that is, for non-local 
boundary conditions. For our subproblems arising later, we also need solvability for such problems. 
To guarantee unique solvability, Zikirov states compatibility conditions (see \cite[Problem 1]{MR2548434}) for the boundary data that read in our case:
\begin{prop}\label{prop:214}
   The third-order problem
%    \begin{equation*}
%       \begin{split}
%       (\boldsymbol{b} \cdot \nabla)\Delta \psi -c \Delta \psi=f \hspace{2em}&\mbox{ in }\Omega,\\
%       \psi(0,y)=\varphi_1(y),\hspace*{2em} \psi(1,y)=\varphi_2(y), \hspace{2em}&0\leq y \leq 1,\\
%       \psi(x,0)=\kappa_1(x), \hspace*{2em}\psi(x,1)=\kappa_2(x),\hspace{2em}&0\leq x \leq 1,\\
%       \partial_x \psi(1,y)=\varphi_3(y),\hspace{2em}&0\leq y \leq 1,\\
%       \partial_y \psi(x,1)=\kappa_3(x),\hspace{2em}&0\leq x \leq 1,
%       \end{split}
%    \end{equation*}
   \begin{gather*}
      (\boldsymbol{b} \cdot \nabla)\Delta \psi -c \Delta \psi=f \hspace{2em}\mbox{ in }\Omega,\\
      \begin{aligned}{4}
         \psi(0,y)&=\varphi_1(y),&
         \psi(1,y)&=\varphi_2(y),& 
         \partial_x \psi(1,y)&=\varphi_3(y),&
         0&\leq y \leq 1,\\
         \psi(x,0)&=\kappa_1(x),& 
         \psi(x,1)&=\kappa_2(x),&
         \partial_y \psi(x,1)&=\kappa_3(x)&
         0&\leq x \leq 1,
      \end{aligned}    
   \end{gather*}
   where $\boldsymbol{b}=(b_1,b_2)$ with $b_1,b_2>0$ and $c\geq 0$ are constant, has a 
   unique solution if the following compatibility conditions are fulfilled:
   \begin{align*}
      \varphi_1(0)&=\kappa_1(0),&\varphi_1(1)&=\kappa_2(0),& \kappa_1(1)&=\varphi_2(0),&\varphi_2(1)&=\kappa_2(1),\\
      \partial_x \kappa_1(1)&=\varphi_3(0),& \partial_y \varphi_1(1)&=\kappa_3(0),& \partial_x \kappa_2(1)&=\varphi_3(1),&\partial_y \varphi_2(1)&=\kappa_3(1).
   \end{align*}
\end{prop}

\section{A stability estimate}\label{sec:stab}
Another important tool for the asymptotic analysis is a stability estimate.
For the statement, we recall $C^{\alpha}(I)$ for any interval 
$I\subseteq \mathbb{R}$, the space of $\alpha$-H\"older continuous functions on $I$ endowed with the usual norm. 
The special cases $\alpha=0$ and $\alpha=1$ yield the space of continuous functions and continuously differentiable 
functions, respectively. Similarly, we make use of fractional order Sobolev spaces $H^\alpha(I)$ on bounded intervals, 
$\alpha\in \mathbb{R}_{\geq 0}$, being the complex interpolation spaces of the respective Sobolev 
spaces with integer values. We recall from \cite{Tr95} the following essential properties of the 
H\"older spaces and fractional order Sobolev spaces.
\begin{theorem}[\!\!{{\cite[Sect 4.5.2 Rem 2, Sect 2.4.2 Rem 2(d), Sect 4.2.2 Thm]{Tr95}}}] 
   Let $I\subseteq \mathbb{R}$ be a bounded closed interval. 
   Then for $\alpha\in (0,1)$ there exist $d_1,d_2>0$ such that
   \[
      \| f\|_{C^\alpha(I)}\leq d_1 \|f\|_{C^0(I)}^{1-\alpha}\|f\|_{C^1(I)}^\alpha
      \quad\text{and}\quad
      \| g\|_{H^\alpha(I)}\leq d_2 \|g\|_{H^0(I)}^{1-\alpha}\|g\|_{H^1(I)}^\alpha
   \]
   for all $f\in C^1(I)$, $g\in H^1(I)$.
\end{theorem}

\begin{defi}
   Consider two functions $g_1 \in C^1(\partial \Omega)\cap H^{3/2}(\partial \Omega)$ 
   (more explicitly $g_1 \in C^1(\Gamma_i)\cap H^{3/2}(\Gamma_i)$, where $\Gamma_i$ are the 
   edges of the unit square) and $g_2\in C(\partial \Omega)\cap H^{1/2}(\partial \Omega)$. 
   We say $(g_1,g_2)$ are \textit{admissible} boundary values, if there exists a function 
   $\phi \in \domain{\Delta^2}$, such that
   \[
      \phi|_{\partial\Omega}=g_1
      \quad\text{and}\quad
      \partial_\boldn \phi|_{\partial \Omega}=g_2.
   \]
\end{defi}
We now consider the following problem
\begin{equation}\label{eq:inhomo}
   \begin{split}
      \varepsilon \Delta^2 \psi + (\boldsymbol{b} \cdot \nabla)\Delta \psi -c \Delta \psi=f \hspace{2em}&\mbox{ in }\Omega=(0,1)^2,\\
      \psi=g_1\hspace{2em}&\mbox{ on }\Gamma=\partial \Omega,\\
      \partial_\boldn \psi=g_2\hspace{2em}&\mbox{ on }\Gamma,
   \end{split}
\end{equation}
where $(g_1,g_2)$ are admissible boundary values and $f\in L^2(\Omega)$. 

In the following, whenever appropriate, we stick to the custom of denoting by $C>0$ a generic constant independent of $\eps$.
\begin{theorem}\label{th:stab2d}
   Consider problem \eqref{eq:inhomo}. The solution $\psi$ can be estimated by
   \[
      \|\psi\|_{L^\infty(\Omega)}\leq C \left(\varepsilon^{-1} \left(\|f\|_{L^2(\Omega)}+\|g_1\|_{C^{1/2}(\partial\Omega)}+\|g_2\|_{L^2(\partial \Omega)}\right)
                                              +\|g_1\|_{C^1(\partial \Omega)}+\|g_2\|_{C(\partial \Omega)}\right).
   \]
   Furthermore, we have
   \[
      \|\psi\|_{H^1(\Omega)}\leq C \varepsilon^{-1} \left(\|f\|_{L^2(\Omega)}+\|g_1\|_{C^{1/2}(\partial\Omega)}+\|g_2\|_{L^2(\partial\Omega)}\right).
   \]
\end{theorem}
Before proving this theorem, we collect some intermediate results in two lemmas.
\begin{lemma}\label{lem:216}
   Consider problem \eqref{eq:inhomo} with homogeneous boundary conditions, that is, assume $g_1=g_2=0$. 
   The solution $u \in H_0^2(\Omega)\cap H^4(\Omega)$ can then be estimated by
   \[
      \|u\|_{H^2(\Omega)}\leq C \varepsilon^{-1}\|f\|_{H^{-2}(\Omega)}.
   \]
\end{lemma}
\begin{proof}
   For $u\in H_0^2(\Omega)\cap H^4(\Omega)$, using integration by parts and $\Div\boldb=0$, we get
   \[
       \Re\langle (\boldb \cdot\nabla) \Delta u,u\rangle=-\Re\langle u,(\boldb \cdot\nabla) \Delta u\rangle,       
   \]
   which yields
   \[
    \Re \langle (\boldb \cdot\nabla) \Delta u,u\rangle=0.
   \]
Thus, we get from the differential equation
  \begin{align*}
       \Re\langle f,u\rangle_{L^2}&=\Re\langle Lu,u\rangle
       \\ & =\Re\langle\left(\varepsilon\Delta^2 u+ (\boldb \cdot\nabla) \Delta u-c \Delta u\right), u\rangle
       \\ & =\varepsilon\|\Delta u\|_{L^2(\Omega)}^2 +c\|\nabla u\|_{L^2(\Omega)}^2.
   \end{align*}
   Hence,
   \[
    \varepsilon\|\Delta u\|_{L^2(\Omega)}^2 +c\|\nabla u\|_{L^2(\Omega)}^2  \leq \|f\|_{H^{-2}(\Omega)}\|u\|_{H^2(\Omega)}.
   \]
%    \begin{align*}
%        \left|\int_\Omega \left(\varepsilon\Delta^2 u+ (\boldb \cdot\nabla) \Delta u-c \Delta u\right) u\right|
%        &=\varepsilon\|\Delta u\|_{L^2(\Omega)}^2 +c\|\nabla u\|_{L^2(\Omega)}^2 ,\\
%        \left|\int_\Omega \left(\varepsilon\Delta^2 u+ (\boldb \cdot\nabla) \Delta u-c \Delta u\right) u\right|
%        &=\left|\int_\Omega Lu\cdot u  \right|
%         =\left|\int_\Omega f u\right|,\\[1em]
%        \Rightarrow\quad
%         \varepsilon\|\Delta u\|_{L^2(\Omega)}^2 +c\|\nabla u\|_{L^2(\Omega)}^2  &\leq \|f\|_{H^{-2}(\Omega)}\|u\|_{H^2(\Omega)}.
%    \end{align*}
   Using the equivalence of the $H^2$-norm and $u\mapsto \|\Delta u\|_{L^2(\Omega)}$  in $H_0^2$ (cf.~Remark \ref{rem:reg_delta_int}), we get the desired result.
\end{proof}
\begin{lemma}\label{lem:217}
   Let $(g_1,g_2)$ be admissible boundary values. Let $v\in H^2(\Omega)$ be the variational solution of the problem
   \begin{equation}\label{eq:inhomobilaplace}
      \begin{split}
      \Delta^2 v=0\hspace{2em}&\mbox{ in }\Omega=(0,1)^2,\\
      v=g_1\hspace{2em}&\mbox{ on }\Gamma=\partial \Omega,\\
      \partial_\boldn v=g_2\hspace{2em}&\mbox{ on }\Gamma.
      \end{split}
   \end{equation}
   The triple $(v,g_1,g_2) \in H^1(\Omega) \times H^{1/2}(\partial \Omega) \times H^{-1/2}(\partial \Omega)$ 
   can also be understood as a generalised solution of \eqref{eq:inhomobilaplace} in the sense of 
   \cite[Theorem 3.2.1]{KMR97}. Then, $v$ fulfils the following estimate
   \[
      \|v\|_{H^1(\Omega)}\leq C \left(\|g_1\|_{H^{1/2}(\partial \Omega)}+\|g_2\|_{H^{-1/2}(\partial \Omega)}\right).
   \]
\end{lemma}
\begin{proof}
   The proof follows an idea of \cite{RPersonal}, which we will repeat here. For ease of formulation, we stick 
   to the case of real-valued functions, the complex case follows similar lines.
   
   Let $w \in H_0^2(\Omega)$ be the variational solution of the problem
   \begin{equation}\label{eq:helpproblem}
      \begin{split}
         \Delta^2 w=\Delta v -v \hspace{2em}&\mbox{ in }\Omega=(0,1)^2,\\
         w=0\hspace{2em}&\mbox{ on }\Gamma=\partial \Omega,\\
         \partial_\boldn w=0\hspace{2em}&\mbox{ on }\Gamma.
      \end{split}
   \end{equation}
   It follows from \cite[Theorem 2]{BR80} that $w \in H^3(\Omega)$ and
   \begin{equation}
      \|w\|_{H^3(\Omega)}\leq C \|\Delta v -v\|_{H^{-1}(\Omega)}\leq C \|v\|_{H^1(\Omega)}.\label{eq:wv}
   \end{equation}
   Furthermore, we find a function $z \in H^1(\Omega)$ with $z=g_1$ on $\partial \Omega$ in the sense of traces and
   \begin{equation}
      \|z\|_{H^1(\Omega)} \leq C \|g_1\|_{H^{1/2}(\partial \Omega)},\label{eq:zg1}
   \end{equation}
   see \cite[Corollary B.53]{EG04}. With the help of Green's formula, we get
   \begin{equation}\label{eq:greenzauber}
      \begin{split}
         0&= \int_\Omega (\Delta^2 v) w 
           =\int_\Omega \Delta v \Delta w 
           =\int_\Omega v \Delta^2 w + \int_{\partial \Omega} \left((\partial_\boldn v) \Delta w -v \partial_\boldn \Delta w\right)\\
          &=\int_\Omega v(\Delta v-v) +\int_{\partial\Omega} g_2 \Delta w -\int_{\partial \Omega} z \partial_\boldn \Delta w.
      \end{split}
   \end{equation}
   We now make use of the equations
   \[
      \int_\Omega v (\Delta v - v)=-\int_{\Omega} \left(|v|^2+|\nabla v|^2\right) +\int_{\partial\Omega} v \partial_\boldn v
   \]
   and
   \begin{align*}
      \int_{\partial \Omega} z \partial_\boldn \Delta w 
         &=\int_{\partial\Omega} z \nabla \Delta w \cdot n 
          =\int_\Omega \Div(z\nabla \Delta w)
          =\int_\Omega (\nabla z \cdot \nabla \Delta w + z \Delta^2 w)\\
         &=\int_\Omega (\nabla z \cdot \nabla \Delta w+z\Delta v-zv)\\
         &=\int_\Omega (\nabla z\cdot \nabla \Delta w-\nabla z\cdot \nabla v-zv)+\int_{\partial\Omega} z \partial_\boldn v
   \end{align*}
   to get from equation \eqref{eq:greenzauber} and the relation $v=z$ on $\partial \Omega$
   \begin{equation}\label{eq:help2}
      \int_\Omega (|v|^2+|\nabla v|^2)=\int_\Omega (zv+\nabla z\cdot \nabla v-\nabla z\cdot \nabla \Delta w)+\int_{\partial\Omega}g_2 \Delta w.
   \end{equation}
   Using \eqref{eq:zg1} we can estimate the first term by
   \begin{align*}
      \left|\int_\Omega (zv+\nabla z\cdot \nabla v-\nabla z\cdot \nabla \Delta w)\right|
         &\leq \|z\|_{H^1(\Omega)}\|v\|_{H^1(\Omega)}+\|z\|_{H^1(\Omega)}\|w\|_{H^3(\Omega)}\\
         &\leq C \|g_1\|_{H^{1/2}(\partial \Omega)}\|v\|_{H^1(\Omega)}
   \end{align*}
   and with the help of a trace inequality (see \cite[Theorem 5.5]{necas}) and \eqref{eq:wv}, 
   we get for the second term
   \begin{align*}
      \left|\int_{\partial\Omega}g_2 \Delta w\right|
         &\leq \|g_2\|_{H^{-1/2}(\partial\Omega)}\|\Delta w\|_{H^{1/2}(\partial\Omega)}
          \leq C \|g_2\|_{H^{-1/2}(\partial\Omega)} \|w\|_{H^3(\Omega)}\\
         &\leq C \|g_2\|_{H^{-1/2}(\partial\Omega)} \|v\|_{H^1(\Omega)}.
   \end{align*}
   From \eqref{eq:help2} we get 
   \[
      \|v\|_{H^1(\Omega)}^2 \leq C \left( \|g_1\|_{H^{1/2}(\partial \Omega)} + \|g_2\|_{H^{-1/2}(\partial\Omega)} \right)\|v\|_{H^1(\Omega)},
   \]
   which we wanted to prove.
\end{proof}
\begin{proof}[Theorem \ref{th:stab2d}]
   We want to have an estimate for the function $\psi$. We consider the function $v$ and $w$ with $\psi=v+w$, where $v$ is 
   the solution of the problem \eqref{eq:inhomobilaplace} and $w$ is the solution of
   \begin{equation*}
      \begin{split}
         Lw=f-(\boldb\cdot \nabla)\Delta v+c\Delta v \hspace{2em}&\mbox{ in }\Omega=(0,1)^2,\\
         w=0\hspace{2em}&\mbox{ on }\Gamma=\partial \Omega,\\
         \partial_\boldn w=0\hspace{2em}&\mbox{ on }\Gamma.
      \end{split}
   \end{equation*}
   By Lemmas \ref{lem:216} and \ref{lem:217} we have
   \begin{align*}
      \|w\|_{H^2(\Omega)}
         &\leq C\varepsilon^{-1}\left( \|f\|_{H^{-2}(\Omega)}+\|-(\boldb\cdot \nabla)\Delta v+c\Delta v\|_{H^{-2}(\Omega)}\right)\\
         &\leq C\varepsilon^{-1}\left(\|f\|_{H^{-2}(\Omega)}+\|v\|_{H^1(\Omega)}\right)\\
         &\leq  C\varepsilon^{-1}\left(\|f\|_{H^{-2}(\Omega)}+\|g_1\|_{H^{1/2}(\partial \Omega)}+\|g_2\|_{H^{-1/2}(\partial \Omega)}\right)\\
         &\leq  C\varepsilon^{-1}\left(\|f\|_{L^2(\Omega)}+\|g_1\|_{C^{1/2}(\partial \Omega)}+\|g_2\|_{L^2(\partial \Omega)}\right).
   \end{align*}
   With the Agmon-Miranda maximum principle, see \cite[Theorem 10]{MR91}, applied to $v$ we obtain
   \begin{equation}\label{eq:agmon}
      \|v\|_{C^1(\bar{\Omega})}\leq C\left(\|g_1\|_{C^1(\partial \Omega)}+\|g_2\|_{C(\partial \Omega)}\right).
   \end{equation}
   By the continuity of the embedding $H^2(\Omega)\hookrightarrow L^\infty(\Omega)$, see \cite[Theorem 4.12]{Adams03}, and \eqref{eq:agmon} we get
   \begin{align*}
      \|\psi\|_{L^\infty(\Omega)}&\leq \|w\|_{L^\infty(\Omega)}+\|v\|_{L^\infty(\Omega)}\\
      &\leq C (\|w\|_{H^2(\Omega)}+\|v\|_{C^1(\Omega)})\\
      &\leq C \left(\varepsilon^{-1} \left(\|f\|_{L^2(\Omega)}+\|g_1\|_{C^{1/2}(\partial\Omega)}+\|g_2\|_{L^2(\partial \Omega)}\right)+\|g_1\|_{C^1(\partial \Omega)}+\|g_2\|_{C(\partial \Omega)}\right).
   \end{align*}
   For the second estimate in Theorem \ref{th:stab2d} we use the $H^1$-norm estimate from Lemma \ref{lem:217} and get
   \begin{align*}
      \|\psi\|_{H^1(\Omega)}
         &\leq \|w\|_{H^1(\Omega)}+\|v\|_{H^1(\Omega)}\\
         &\leq C(\|w\|_{H^2(\Omega)}+\|v\|_{H^1(\Omega)})\\
         &\leq C \varepsilon^{-1} \left(\|f\|_{L^2(\Omega)}+\|g_1\|_{C^{1/2}(\partial\Omega)}+\|g_2\|_{L^2(\partial\Omega)}\right).
   \end{align*}
\end{proof}
\section{Solution decomposition}\label{sec:decomposition}

For the derivation of the solution decomposition for the 2D problem \eqref{eq:fourthorder}, 
we use the method of asymptotic expansions in powers of $\eps$, see for instance \cite{JF96}. 
We will follow ideas presented in \cite{LS01}, where an asymptotic expansion of the type $\sum_{i} \varepsilon^i u_i$
for a second-order problem is given. There are some differences between the approach presented in \cite{LS01} and ours, mainly
due to the fact that we have a fourth-order problem and two different types of boundary conditions.
In particular, the Neumann-boundary condition necessitates correction functions that interact
across different $\eps$-levels. Thus, in comparison to the second-order problem, the structure is more involved.

The boundary layer terms involve exponentially decaying functions in $x$ and $y$, which we will denote by
\[
 E_1(x) = \e^{-b_1 x/ \varepsilon}
 \quad\mbox{and}\quad
 E_2(y) = \e^{-b_2 y/ \varepsilon}.
\]
In this section, we provide a formal analysis and assume all solutions to be as smooth as needed. 
In fact, $C^4(\bar{\Omega})$ will be sufficient.

\subsection{Formal expansion}
The formal ansatz as an infinite series for the structure of the solution $\psi$ would be:
\begin{equation}\label{eq:ansatz2d}
   \psi=\sum_{i=0}^\infty \varepsilon^i \psi_i +\sum_{i=0}^\infty \varepsilon^i v_i+\sum_{i=0}^\infty \varepsilon^i w_i +\sum_{i=0}^\infty \varepsilon^i z_i.
\end{equation} 
Since we are interested in a lower order expansion only, we will confine ourselves with finite 
sums in the expression \eqref{eq:ansatz2d}. To be more precise, we seek an approximation of $\psi$, 
the solution of \eqref{eq:fourthorder}, of the form
\begin{equation}\label{eq:ansatzfinite}
   \Psi=  \sum_{i=0}^j \varepsilon^i \psi_i 
         +\sum_{i=0}^k \varepsilon^i v_i
         +\sum_{i=0}^l \varepsilon^i w_i
         +\sum_{i=0}^m \varepsilon^i z_i,
\end{equation} 
where the integers $j,k,l,m$ will be specified later. The first sum represents the outer expansion 
by means of reduced problems with Dirichlet boundary conditions and homogeneous Neumann conditions at the 
inflow boundary. The boundary corrections needed are corrections of Neumann values. They are split into two 
sums for the different sides of $\Omega$ in terms near the outflow corner. The final sum of \eqref{eq:ansatzfinite} 
corrects Neumann values introduced by the previous correction terms near the outflow corner. 
Figure~\ref{fig:corrections} depicts the situation for the first step for $\boldb=(b_1,b_2)\in \mathbb{R}^2$ 
with $b_1,b_2>0$, which we shall abbreviate by writing $\boldb>0$.

\begin{figure}[htb!]
\begin{center}
\includegraphics{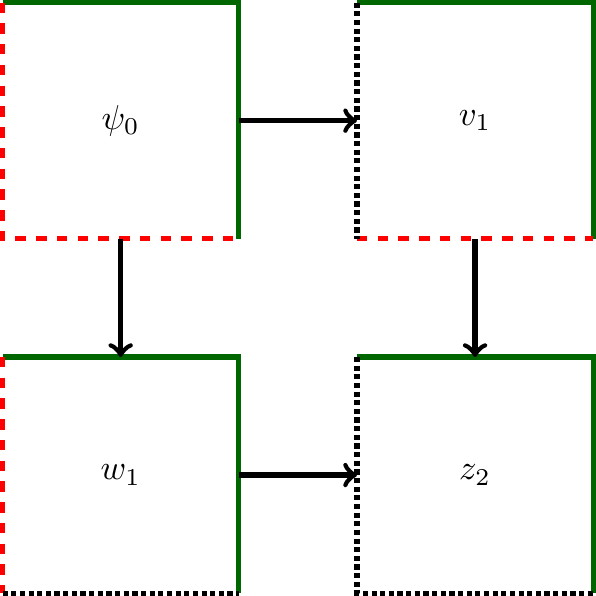}
\end{center}
\caption{\label{fig:corrections} The interplay of Neumann corrections, solid line: no error, dashed line: error, densely-dotted line: correction}
\end{figure}
Let us start by defining $\psi_0$. It solves a reduced problem that can formally be found by 
comparing powers of $\varepsilon$ in the general approach \eqref{eq:ansatz2d}. We have
\begin{subequations}\label{eq:psi_0}
 \begin{alignat}{2}
   (\boldsymbol{b}\cdot \nabla)\Delta \psi_0-c\Delta \psi_0&=f\quad&&\mbox{ in }\Omega=(0,1)^2,\\
   \psi_0&=0 &&\mbox{ on }\Gamma=\partial \Omega,\\
  \partial_\boldn \psi_0&=0 &&\mbox{ on }\Gamma_-,
\end{alignat}
\end{subequations}
where $\Gamma_-:=\{(x,y)\in \Gamma|-\boldsymbol{b}\cdot \boldsymbol{n}(x,y)<0\}$ is the inflow boundary.
According to Proposition \ref{prop:214} the third-order problem \eqref{eq:psi_0} has a 
unique solution satisfying homogeneous Dirichlet conditions on the whole boundary $\Gamma$, 
and homogeneous Neumann conditions on the inflow boundary $\Gamma_-$ only. This is similar to the 
asymptotic expansions for second order problems as in \cite{LS01}.

Now the normal derivative of $\psi_0$ is in general non-zero on the outflow boundary. 
Thus, we correct these Neumann data by correction functions $v_i$ in $x$-direction and 
$w_i$ in $y$-direction. We recall our standing assumption of $\boldb>0$, the general case $\boldb=(b_1,b_2)$ with $b_1\cdot b_2\neq 0$ can be derived from the one with $\boldb>0$ by a change of coordinates. Using $\boldb>0$, we find the outflow boundary at $x=0$ and $y=0$. The layers occur at $x=0$ and $y=0$ as well.

We will now define the correction function for the layer at $x=0$. They can formally be
derived with the help of the stretched variable $\xi$ defined by $\xi=\frac{x}{\varepsilon}$ 
and comparison of powers of $\varepsilon$ in the transformed differential operator $\tilde{L}$. 
We obtain $\tilde{L}$ expressed in terms of the variable $\xi$ and $y$ by a coordinate transformation
\begin{align}
   \tilde{L}\tilde{v}(\xi,y)
      &=   \varepsilon^{-3}\tilde{v}_{\xi\xi\xi\xi}
         +2\varepsilon^{-1}\tilde{v}_{\xi\xi yy}
         +\varepsilon \tilde{v}_{yyyy}
         +\varepsilon^{-3}b_1\tilde{v}_{\xi\xi\xi}
         +\varepsilon^{-1}b_1\tilde{v}_{\xi yy}
         +\varepsilon^{-2}b_2\tilde{v}_{\xi\xi y}
         +b_2\tilde{v}_{yyy}\nonumber\\
         &\quad-\varepsilon^{-2}c\tilde{v}_{\xi\xi}
         -c\tilde{v}_{yy}\nonumber\\
     &= \varepsilon^{-3}(\tilde{v}_{\xi}+b_1\tilde{v})_{\xi\xi\xi}
       +\varepsilon^{-2}(b_2\tilde{v}_{y}-c\tilde{v})_{\xi\xi}
       +\varepsilon^{-1}(2\tilde{v}_{\xi}+b_1\tilde{v})_{\xi yy}
       +(b_2\tilde{v}_{y}-c\tilde{v})_{yy}
       +\varepsilon \tilde{v}_{yyyy}.\label{eq:ltilde}
\end{align}
Now the function $v$ is to correct the Neumann data. Thus, it has to fulfil the boundary conditions 
at $x=0$, that is $\xi=0$. Therefore, we consider the Neumann derivative in $x=0$ of the (formally) infinite 
sum \eqref{eq:ansatz2d} and make a comparison in the powers of $\varepsilon$. The last two sums are 
set to zero, because they do not have to correct anything. We get conditions on the functions $\psi_i$ 
and $\tilde{v}_i$, namely
\begin{equation}\label{eq:NeumCorrec}
   \psi_{i,x}(0,y) =- \tilde{v}_{i+1,\xi}(0,y)
\end{equation}
for correcting the contributions of $\psi_i$ to the Neumann data at $x=0$. The functions $\psi_i$ 
and $\tilde{v}_i$ act on different $\varepsilon$-levels because we consider the derivative of 
$\tilde{v}_i$ in $\xi$ and thus, we get an additional order of $\varepsilon$. Note that we do not 
correct the boundary conditions up to arbitrarily large values of $i$. We make use of this condition 
only for the considered finite number of indices. Additionally, the boundary correction functions 
are expected to be exponentially decaying away from the boundary. For this reason, we expand the domain of $\tilde{L}$ from $(0,\frac{1}{\varepsilon})\times (0,1)$ to  $(0,\infty)\times (0,1)$.

Now, the boundary condition \eqref{eq:NeumCorrec} shows that the contribution of $\psi_0$ will 
be corrected by $v_1$. Thus, $v_0$ has nothing to correct and therefore we set
\[
   \tilde{v}_0 = 0.
\]
For $\tilde{v}_1$ the comparison of powers of $\varepsilon$ in \eqref{eq:ltilde} with $\tilde{L}\tilde{v}=0$ yields
\begin{equation*}
   \begin{split}
      \tilde{v}_{1,\xi\xi\xi\xi}+b_1\tilde{v}_{1,\xi \xi\xi}=-b_2\tilde{v}_{0,\xi\xi y}+c\tilde{v}_{0,\xi\xi}&=0 \hspace*{2em}\mbox{ in }(0,\infty)\times(0,1),\\
      \tilde{v}_{1,\xi}(0,y)=-\psi_{0,x}(0,y) \mbox{ and } \lim_{\xi \rightarrow \infty} \tilde{v}_1(\xi,y)&=0.
   \end{split}
\end{equation*}
This problem has constant coefficients and no derivatives in $y$. 
Therefore, it can be solved explicitly and has the solution
\[
   \tilde{v}_1(\xi,y)=\frac{\psi_{0,x}(0,y)}{b_1}\exp\{-b_1\xi\} 
   \quad\text{ or }\quad
   v_1(x,y)=\frac{\psi_{0,x}(0,y)}{b_1}E_1(x).
\]
Analogously, we obtain the correction function $\tilde{w}_0=0$ and $\tilde{w}_1(x,\eta)$ 
with the stretched variable $\eta=\frac{y}{\varepsilon}$ for the layer along $y=0$:
\[
   \tilde{w}_1(x,\eta)=\frac{\psi_{0,y}(x,0)}{b_2}\exp\{-b_2\eta\} 
   \quad\text{ or }\quad
   w_1(x,y)=\frac{\psi_{0,y}(x,0)}{b_2}E_2(y).
\]
Both boundary correction functions $v_1$ and $w_1$ introduce non-zero Dirichlet and Neumann contributions 
at the outflow-boundary. We correct the Neumann data by a corner correction function $z_2$ and the Dirichlet 
data by $\psi_1$. For the corner-correction function we apply the stretching of the coordinates in both 
directions and obtain formally the operator
\[
   \bar{L}\bar{z}(\xi,\eta)
      =   \varepsilon^{-3}(\bar{z}_{\xi\xi\xi\xi}+2\bar{z}_{\xi\xi \eta \eta}+\bar{z}_{\eta \eta\eta\eta})
         +\varepsilon^{-3}(b_1(\bar{z}_{\xi\xi\xi}+\bar{z}_{\eta\eta\xi})+b_2(\bar{z}_{\xi\xi\eta}+\bar{z}_{\eta\eta\eta}))
         -\varepsilon^{-2}c(\bar{z}_{\xi\xi}+\bar{z}_{\eta\eta}).
\]
Again, the correction function is to correct the Neumann data of $v$ along $y=0$ ($\eta=0$)
and of $w$ along $x=0$ ($\xi=0$). Thus we have
\[
   \tilde{v}_{i,y}(\xi,0)=-\bar{z}_{i+1,\eta}(\xi,0)
   \quad\mbox{ and } \quad
   \tilde{w}_{i,x}(0,\eta)=-\bar{z}_{i+1,\xi}(0,\eta),
\]
respectively. This time we obtain $\bar z_0= \bar z_1= 0$ because they have nothing to correct.
The corner-correction function $\bar{z}_2$ satisfies
\begin{align*}
   \bar{\Delta}^2 \bar{z}_2 +(b\cdot \bar{\nabla})\bar{\Delta} \bar{z}_2
       = c\bar{\Delta} \bar{z}_1
      &=0\hspace*{6em}\mbox{in }(0,\infty)\times(0,\infty),\\
   \bar{z}_{2,\eta}(\xi,0)
       =-\tilde{v}_{1,y}(\xi,0) \mbox{ and }
   \bar{z}_{2,\xi}(0,\eta)
      &=-\tilde{w}_{1,x}(0,\eta),&\\
   \lim_{\xi \rightarrow \infty}\bar{z}_2(\xi,\eta)=0 \mbox{ and }
   \lim_{\eta \rightarrow \infty}\bar{z}_2(\xi,\eta)&=0.&
\end{align*}
A solution can be found to be 
\[
   \bar{z}_2(\xi,\eta)=-\frac{\psi_{0,xy}(0,0)}{b_1b_2}\exp\{-b_1\xi\}\exp\{-b_2\eta\}
   \quad\text{ or }\quad
   z_2(x,y)=-\frac{\psi_{0,xy}(0,0)}{b_1b_2}E_1(x)E_2(y).
\]

As said before, the boundary-correction functions $v_1$ and $w_1$ introduce non-neglectable 
contributions in the Dirichlet data on $\Gamma_+:=\{(x,y)\in \Gamma: -\boldb \cdot \boldsymbol{n}(x,y) >0\}$. 
This will be corrected in the next step by $\psi_1$ satisfying the reduced problem
\begin{alignat*}{2}
   (\boldsymbol{b}\cdot \nabla)\Delta \psi_1-c\Delta \psi_1&=-\Delta^2 \psi_0&&\hspace*{2em}\mbox{in }\Omega=(0,1)^2,\\
   \psi_1&=0 &&\hspace*{2em}\mbox{on }\Gamma_-,\\
   \partial_\boldn \psi_1(x,y)&=0 &&\hspace*{2em}\mbox{on }\Gamma_-,\\
   \psi_1(0,y)&=-v_1(0,y)&&\hspace*{2em} y\in (0,1), \\
   \psi_1(x,0)&=-w_1(x,0)&&\hspace*{2em} x\in(0,1).
\end{alignat*}
It can be checked that the conditions for existence of a unique solution, given in Proposition \ref{prop:214}, are fulfilled.

Now the construction of problems for $v_2$ and $w_2$ follows the same pattern as the construction for $v_1$ and $w_1$, respectively. We get
\begin{align*}
   \tilde{v}_2(\xi,y) &= \left(\frac{\psi_{1,x}(0,y)}{b_1}-\frac{\alpha(y)}{b_1^3}-\frac{\alpha(y) \xi}{b_1^2}\right)\exp\{-b_1\xi\}, \\
   \text{ that is }v_2(x,y)&=\left(\frac{\psi_{1,x}(0,y)}{b_1}-\frac{\alpha(y)}{b_1^3}-\frac{\alpha(y) x}{b_1^2 \varepsilon}\right)E_1(x)\\
   \mbox{with }\alpha(y)&=-b_2\psi_{0,xy}(0,y)+c\psi_{0,x}(0,y)
   \intertext{and}
   \tilde{w}_2(x,\eta) &= \left(\frac{\psi_{1,y}(x,0)}{b_2}-\frac{\beta(x)}{b_2^3}-\frac{\beta(x) \eta}{b_2^2}\right)\exp\{-b_2\eta\}, \\
   \text{ that is }w_2(x,y)&=\left(\frac{\psi_{1,y}(x,0)}{b_2}-\frac{\beta(x)}{b_2^3}-\frac{\beta(x) y}{b_2^2 \varepsilon}\right)E_2(y)\\
   \mbox{with }\beta(x)&=-b_1\psi_{0,xy}(x,0)+c\psi_{0,y}(x,0).
   \intertext{The function $\bar{z}_3$ has to fulfil the following problem:}
    \bar{\Delta}^2 \bar{z}_3 +(b\cdot \bar{\nabla})\bar{\Delta} \bar{z}_3
         =c\bar{\Delta} \bar{z}_2
        &=-\frac{c(b_1^2+b_2^2)}{b_1b_2} \psi_{0,xy}(0,0)\exp\{-b_1\xi\}\exp\{-b_2\eta\}&\mbox{in }(0,\infty)\times(0,\infty),\\
    \bar{z}_{3,\eta}(\xi,0)&=-\tilde{v}_{2,y}(\xi,0) 
    \quad\mbox{and}\quad
    \bar{z}_{3,\xi}(0,\eta) =-\tilde{w}_{2,x}(0,\eta),&\\
    \lim_{\xi \rightarrow \infty}\bar{z}_3(\xi,\eta)&=0 
    \quad\mbox{and}\quad
    \lim_{\eta \rightarrow \infty}\bar{z}_3(\xi,\eta)=0.&
\end{align*}
We make the ansatz
\[
   {z}_3(\xi,\eta)=(\omega_1+\omega_2 \xi + \omega_3 \eta) \exp\{-b_1 \xi \}\exp\{-b_2 \eta\}
\]
with unknown constants $\omega_1, \omega_2,\omega_3$.  With this ansatz, we get 
only a solution if the compatibility condition
\begin{equation}\label{eq:compcond1}
   (\boldb \cdot \nabla) \psi_{0,xy}(0,0)-c\psi_{0,xy}(0,0)=0
\end{equation}
holds true.
Then, the solution is given by
\[
   \bar{z}_3(\xi,\eta)=(\omega_1+\omega_2 \xi + \omega_3 \eta) \exp\{-b_1 \xi \}\exp\{-b_2 \eta\}
   \quad\text{ or }\quad
   z_3(x,y)=\left(\omega_1+\omega_2 \frac{x}{\varepsilon} + \omega_3 \frac{y}{\varepsilon}\right) E_1(x) E_2(y) 
\]
with
\begin{align*}
   \omega_1&=\frac{-b_1 b_2\psi_{1,xy}(0,0) +b_1\psi_{0,xyy}(0,0) + b_2\psi_{0,xxy}(0,0)}{b_1^2 b_2^2},\\
   \omega_2&=\frac{\psi_{0,xxy}(0,0)}{b_1 b_2}
   \quad\mbox{and}\quad
   \omega_3 =\frac{\psi_{0,xyy}(0,0)}{b_1 b_2}.
\end{align*}
Without the explicit ansatz for $z_3$ we also get several compatibility conditions, 
which follow from the differential equation itself and the boundary conditions, 
warranting $z_3$ to be continuous in the corner (0,0).

Now $z_2$ introduces non-neglectable contributions in the Dirichlet data on $\Gamma_+$. 
Thus for the next step, $\psi_2$ has to correct the Dirichlet data of $v_2,w_2$ and $z_2$. 
We consider the equation
\begin{align*}
   (\boldsymbol{b}\cdot \nabla)\Delta \psi_2-c\Delta \psi_2&=-\Delta^2 \psi_1&&\hspace*{2em}\mbox{in }\Omega=(0,1)^2,\\
   \psi_2&=0 &&\hspace*{2em}\mbox{on }\Gamma_-,\\
   \partial_\boldn \psi_2&=0 &&\hspace*{2em}\mbox{on }\Gamma_-,\\
   \psi_2(0,y)&=-v_2(0,y)-w_2(0,y)-z_2(0,y)&&\hspace*{2em} y\in (0,1),\\
   \psi_2(x,0)&=-v_2(x,0)-w_2(x,0)-z_2(x,0)&&\hspace*{2em} x\in (0,1).
\end{align*}
We again have to check the conditions of Proposition \ref{prop:214}. 
Therefore, we have to verify, whether the following equations hold
\begin{subequations}
   \begin{align}
      \label{eq:cond1}-v_2(0,1)-w_2(0,1)-z_2(0,1)&=0,\\
      \label{eq:cond2}-v_2(1,0)-w_2(1,0)-z_2(1,0)&=0,\\
      \label{eq:cond3}-v_{2,y}(0,1)-w_{2,y}(0,1)-z_{2,y}(0,1)&=0,\\
      \label{eq:cond4}-v_{2,x}(1,0)-w_{2,x}(1,0)-z_{2,x}(1,0)&=0.
   \end{align}
\end{subequations}
We begin with equation \eqref{eq:cond1}:
\begin{align*}
   v_2(0,1)&+w_2(0,1)+z_2(0,1)\\
   &=\frac{\psi_{1,x}(0,1)}{b_1}-\frac{\alpha(1)}{b_1^3}+\left(\frac{\psi_{1,y}(0,0)}{b_2}-\frac{\beta(0)}{b_2^3}-\frac{\beta(0)}{b_2^2 \varepsilon}\right)E_2(1)-\frac{\psi_{0,xy}(0,0)}{b_1 b_2}E_2(1)\\
%    &\quad=\left( \frac{\psi_{1,y}(0,0)}{b_2}-\frac{-b_1 \psi_{0,xy}(0,0)+c\psi_{0,y}(0,0)}{b_2^3}-\frac{-b_1 \psi_{0,xy}(0,0)+c\psi_{0,y}(0,0)}{b_2^2 \varepsilon}-\frac{\psi_{0,xy}(0,0)}{b_1 b_2}\right)E_2(1)\\
   &=\left( -\frac{2}{b_1 b_2}+\frac{b_1 }{b_2^3}+\frac{b_1 }{b_2^2 \varepsilon}\right)\psi_{0,xy}(0,0)E_2(1)
    =0.
\end{align*}
This is zero, if 
\begin{equation}\label{eq:compcond2}
   \psi_{0,xy}(0,0)=0.
\end{equation} 
Equation \eqref{eq:cond2} yields
\begin{align*}
   v_2(1,0)&+w_2(1,0)+z_2(1,0)\\
   &=\left(\frac{\psi_{1,x}(0,0)}{b_1}-\frac{\alpha(0)}{b_1^3}-\frac{\alpha(0)}{b_1^2 \varepsilon}\right)E_1(1)+\frac{\psi_{1,y}(1,0)}{b_2}-\frac{\beta(1)}{b_2^3}-\frac{\psi_{0,xy}(0,0)}{b_1 b_2}E_1(1)\\
%    &\quad=\left( \frac{\psi_{1,x}(0,0)}{b_1}-\frac{-b_2\psi_{0,xy}(0,0)+c\psi_{0,x}(0,0)}{b_1^3}-\frac{-b_2\psi_{0,xy}(0,0)+c\psi_{0,x}(0,0)}{b_1^2 \varepsilon} -\frac{\psi_{0,xy}(0,0)}{b_1 b_2}\right)E_1(1)\\
   &=\left( -\frac{2}{b_1 b_2}+\frac{b_2}{b_1^3}+\frac{b_2}{b_1^2 \varepsilon} \right)\psi_{0,xy}(0,0)E_1(1)
    =0
\end{align*}
and also gives condition \eqref{eq:compcond2}.
Let us continue with equation \eqref{eq:cond3}:
\begin{align*}
   v_{2,y}(0,1)&+w_{2,y}(0,1)+z_{2,y}(0,1)\\
   &=\frac{\psi_{1,xy}(0,1)}{b_1}-\frac{\alpha_y(1)}{b_1^3}+\left(\frac{-b_2}{\varepsilon}\left(\frac{\psi_{1,y}(0,0)}{b_2}-\frac{\beta(0)}{b_2^3}-\frac{\beta(0)}{b_2^2 \varepsilon}\right) -\frac{\beta(0)}{b_2^2 \varepsilon}\right)E_2(1)+\frac{\psi_{0,xy}(0,0)}{b_1 }E_2(1)\\
   &\overset{\eqref{eq:compcond2}}{\underset{}{=}}\frac{\psi_{1,xy}(0,1)}{b_1}-\frac{\alpha_y(1)}{b_1^3}- \frac{\psi_{1,y}(0,0)}{\varepsilon} E_2(1)\\
   &=\frac{b_2 \psi_{0,xyy}(0,1)}{b_1^3}
    =0.
\end{align*}
Analogously, we get for equation \eqref{eq:cond4}
\[
   v_{2,x}(1,0)+w_{2,x}(1,0)+z_{2,x}(1,0)
      = \frac{b_1 \psi_{0,xxy}(1,0)}{b_2^3}
      = 0.
\]
Thus we obtain the additional conditions
\begin{equation}
   \psi_{0,xyy}(0,1)=0
   \quad\mbox{and}\quad
   \psi_{0,xxy}(1,0)=0.\label{eq:compcond3} 
\end{equation}
If the conditions \eqref{eq:compcond2} and \eqref{eq:compcond3} are violated, Proposition \ref{prop:214} is not applicable. Hence, it is unclear, whether a solution $\psi_2$ exists. Condition \eqref{eq:compcond2} implies that $z_2=0$. 
\begin{remark}From the conditions  \eqref{eq:compcond3} we get conditions on $f$, namely 
   $f(0,1)=0$ and $f(1,0)=0$. This follows by extension of the differential equation for $\psi_0$ 
   to the boundary.
   Furthermore, if $b_1=b_2$, condition \eqref{eq:compcond1} implies that $f(0,0)=0$.
\end{remark}
The construction of $v_3$ and $w_3$ follows the same pattern as those for $v_1, v_2, w_1$ and $w_2$. 
The structure of the solutions is of the form 
\begin{align*}
   \tilde{v}_3(\xi,y)=\widetilde{P}_2(\xi,y)\exp\{-b_1\xi\}   
   &\quad\text{ or }\quad 
   v_3(x,y)=\widetilde{P}_2\left(\frac{x}{\varepsilon},y\right)E_1(x),\\
   \tilde{w}_3(x,\eta)=\widetilde{Q}_2(x,\eta)\exp\{-b_2\eta\}
   &\quad\text{ or }\quad
   w_3(x,y)=\widetilde{Q}_2\left(x,\frac{y}{\varepsilon}\right)E_2(y),
\end{align*}
where $\widetilde{P}_2$ is a polynomial of second order in the first variable and $\widetilde{Q}_2$ 
a polynomial of second order in the second variable.
To continue the expansion with higher order terms of $\epsilon$ in the same manner results in more compatibility conditions. So, we stop the classical asymptotic expansion here. We add functions $\tilde{v}_4,\tilde{w}_4$ and $\bar{z}_4$ not correcting the boundary layers, but modifying the right hand-side of the residual. 
We want them to satisfy:
\begin{align*}
   \tilde{v}_{4,\xi\xi\xi\xi}+b_1\tilde{v}_{4,\xi\xi\xi}
    &= -b_2\tilde{v}_{3,\xi\xi y}+c\tilde{v}_{3,\xi\xi}-2\tilde{v}_{2,\xi\xi y y}-b_1\tilde{v}_{2,\xi y y}-b_2 \tilde{v}_{1,yyy}+c\tilde{v}_{1,yy},
     \quad \mbox{in }(0,\infty)\times(0,1),\\
   \tilde{v}_{4,\xi}(0,y) &= 0,\,\mbox{in }(0,1),\\
   \lim_{\xi \rightarrow \infty}\tilde{v}_4(\xi,y) &= 0,\,\mbox{in }(0,1),\\
%  \intertext{and}
   \tilde{w}_{4,\eta\eta\eta\eta}+b_2\tilde{w}_{4,\eta\eta\eta}
    &= -b_1\tilde{w}_{3,x\eta\eta}+c\tilde{w}_{3,\eta\eta}-2\tilde{w}_{2,x x \eta \eta}-b_2\tilde{w}_{2,x x \eta}-b_1 \tilde{w}_{1,xxx}+c\tilde{w}_{1,xx},
     \quad\mbox{in }(0,1)\times(0,\infty),\\
   \tilde{w}_{4,\eta}(x,0) &= 0,\,\mbox{in }(0,1),\\
   \lim_{\eta \rightarrow \infty}\tilde{w}_4(x,\eta)&=0,\,\mbox{in }(0,1),\\
%  \intertext{and}
 \bar{\Delta}^2\bar{z}_4+(b\cdot \bar{\nabla})\bar{\Delta}\bar{z}_4&=c\bar{\Delta}\bar{z}_3.
\end{align*}
The boundary data of $\bar{z}_4$ only have to be bounded by a constant and need not to correct other data, 
thus we have some freedem in the choice for $\bar{z}_4$. We get
\begin{align*}
   \tilde{v}_4(\xi,y)=\widetilde{P}_3(\xi,y)\exp\{-b_1\xi\} 
   &\quad\text{ or }\quad
   v_3(x,y)=\widetilde{P}_3\left(\frac{x}{\varepsilon},y\right)E_1(x),\\
   \tilde{w}_4(x,\eta)=\widetilde{Q}_3(x,\eta)\exp\{-b_2\eta\}
   &\quad\text{ or }\quad
   w_3(x,y)=\widetilde{Q}_3\left(x,\frac{y}{\varepsilon}\right)E_2(y),
\end{align*}
where $\tilde{P}_3$ is a polynomial of third order in the first variable and $\tilde{Q}_3$ a polynomial 
of third order in the second variable. The function $z_4$ is of the structure
\[
   \bar{z}_4(\xi,\eta)=\bar{R}_2(\xi,\eta)\exp\{-b_1\xi\}\exp\{-b_2\eta\}
   \quad\text{ or }\quad
   z_4(x,y)=\bar{R}_2\left(\frac{x}{\varepsilon},\frac{y}{\varepsilon}\right)E_1(x)E_2(y),
\]
where $\bar{R}$ is a polynomial of second order in both variables.

\subsection{Estimating the residual}
Consider now the approximation $\Psi$ of $\psi$ given by
\[
   \Psi=\sum_{i=0}^2 \varepsilon^i \psi_i+\sum_{i=1}^4 \varepsilon^i v_i +\sum_{i=1}^4  \varepsilon^i w_i +\sum_{i=2}^4 \varepsilon^i z_i,
\]
with the functions given before. Furthermore, let $R=\psi-\Psi$ be the residual of the solution of 
problem \eqref{eq:fourthorder} and its approximation $\Psi$. We now estimate the residual $R$ with 
the help of Theorem \ref{th:stab2d}. Let us start with some lemmas. In all the lemmas, it is implicitly 
assumed that the compatibility conditions \eqref{eq:compcond1}, \eqref{eq:compcond2} and \eqref{eq:compcond3} 
are fulfilled. We also recall our standing assumption on the solutions to be sufficiently smooth.
\begin{lemma}\label{lem:LR}
   For the residual $R$ holds
   \begin{equation}\label{eq:L2_LR}
      \|LR\|_{L^2(\Omega)}\leq C\varepsilon^{5/2}.
   \end{equation}
\end{lemma}
\begin{proof}
   Let us start by computing
   \begin{align}
    LR&= L(\psi-\psi_0-\varepsilon\psi_1-\varepsilon^2\psi_2-\varepsilon v_1-\varepsilon^2 v_2-\varepsilon^3 v_3-\varepsilon^4 v_4-\varepsilon w_1-\varepsilon^2 w_2-\varepsilon^3 w_3-\varepsilon^4 w_4\notag\\
    &\quad-\varepsilon^2 z_2-\varepsilon^3 z_3-\varepsilon^4 z_4)\notag\\
      &= L(\psi-\psi_0-\varepsilon\psi_1-\varepsilon^2\psi_2)-L(\varepsilon v_1+\varepsilon^2 v_2+\varepsilon^3 v_3+\varepsilon^4 v_4)\notag\\
       &\quad -L(\varepsilon w_1+\varepsilon^2 w_2+\varepsilon^3 w_3+\varepsilon^4 w_4)
    -L(\varepsilon^2 z_2+\varepsilon^3 z_3+\varepsilon^4 z_4).\label{eq:LR}
   \end{align}
   We obtain for the first term of \eqref{eq:LR}
   \begin{align*}
    L(\psi-\psi_0-\varepsilon\psi_1-\varepsilon^2 \psi_2)
     &= f-f+\varepsilon \Delta^2 \psi_0-\varepsilon \Delta^2 \psi_0+\varepsilon^2 \Delta^2 \psi_1-\varepsilon^2\Delta^2\psi_1+\varepsilon^3 \Delta^2 \psi_2\\
     &= \varepsilon^3 \Delta^2 \psi_2.
   \end{align*}
   For the second term of \eqref{eq:LR} we have
   \begin{align*}
   L(\varepsilon v_1+\varepsilon^2 v_2+\varepsilon^3 v_3+\varepsilon^4 v_4)
    &=\tilde{L}(\varepsilon\tilde{v}_1+\varepsilon^2\tilde{v}_2+\varepsilon^3\tilde{v}_3+\varepsilon^4 \tilde{v}_4)\\
    &=\varepsilon^2\left((\tilde{v}_{1,yy}+b_2\tilde{v}_{2,y}-c\tilde{v}_{2}+2\tilde{v}_{3,\xi\xi}+b_1\tilde{v}_{3,\xi})_{yy}+(b_2\tilde{v}_{4,y}-c\tilde{v}_4)_{\xi\xi}\right)\\
    &\quad
           +\varepsilon^3(\tilde{v}_{2,yy}+b_2\tilde{v}_{3,y}-c\tilde{v}_{3}+2\tilde{v}_{4,\xi\xi}+b_1\tilde{v}_{4,\xi})_{yy}\\
           &\quad
           +\varepsilon^4(\tilde{v}_{3,yy}+b_2\tilde{v}_{4,y}-c\tilde{v}_4)_{yy}+\varepsilon^5 \tilde{v}_{4,yyyy}.
   \end{align*}
   The third term of \eqref{eq:LR} can be rewritten analogously.
   Finally, the remaining term in \eqref{eq:LR} is
   \begin{align*}
    L(\varepsilon^2 z_2+\varepsilon^3 z_3+\varepsilon^4 z_4)
      &= \bar{L}(\varepsilon^2 \bar{z_2}+\varepsilon^3 \bar{z}_3+\varepsilon^4 \bar{z}_4)\\
      &= -\varepsilon^2 c \bar{\Delta} \bar{z}_4.
   \end{align*}
 Thus, we obtain with some functions $g_1,\,g_2,\,g_3$ that are bounded independently of $\eps$
   \[
    LR(x,y) = \eps^3\Delta^2\psi_2(x,y)+\eps^2(g_1(y)E_1(x)+g_2(x)E_2(y)+g_3(x,y)E_1(x)E_2(y))\quad(0\leq x,y\leq 1),
   \]
   which yields \eqref{eq:L2_LR}. 
\end{proof}
\begin{lemma}\label{lem:Rrand}
   The residual $R$ can be bounded on the boundary $\Gamma=\partial \Omega$ by
   \begin{equation}\label{eq:linftyR}
      \|R\|_{L^\infty(\Gamma)}\leq C \varepsilon^3.
   \end{equation}
   Its tangential derivative $\partial_t$ is bounded on $\partial \Omega$ by
   \begin{equation}\label{eq:linftyRt}
      \|\partial_t R\|_{L^\infty(\Gamma)}\leq C \varepsilon^2.
   \end{equation}
\end{lemma}
\begin{proof}
   We will give details only for some of the terms occuring in the expression for $R$ as the estimates for the others follow the same rationale. Let us start out with
   \begin{align*}
      R(1,y)&=\psi(1,y)-\psi_0(1,y)-\varepsilon \psi_1(1,y)-\varepsilon^2 \psi_2(1,y)\\
      &\quad-\varepsilon (v_1(1,y)+w_1(1,y))-\varepsilon^2(v_2(1,y)+w_2(1,y)+z_2(1,y))\\
      &\quad -\varepsilon^3(v_3(1,y)+w_3(1,y)+z_3(1,y))-\varepsilon^4(v_4(1,y)+w_4(1,y)+z_4(1,y))
   \end{align*}
   The first four terms of $R(1,y)$ are zero, because these functions fulfil homogeneous Dirichlet boundary conditions. 

   The correction terms with index 1 yield
   \[
      v_1(1,y)+w_1(1,y)
         =\frac{\psi_{0,x}(0,y)}{b_1}E_1(1)+\frac{\psi_{0,y}(1,0)}{b_2}E_2(0)
         =\frac{\psi_{0,x}(0,y)}{b_1}E_1(1)
   \]
   because $\psi_{0,y}(1,0)=0$ due to homogeneous Dirichlet boundary conditions of $\psi_0$. 
   The factor $E_1(1)$ is exponentially small in $\varepsilon$ and can therefore be 
   bounded by an arbitrary power of $\varepsilon$. For the corrections with index 2 we have
   \[
      w_2(1,y)
         =\left(\frac{\psi_{1,y}(1,0)}{b_2}-\frac{\beta(1)}{b_2^3}-\frac{\beta(1) y}{b_2^2 \varepsilon}\right)E_2(y)
         =0
   \]
   because $\beta(1)=0$ due to $\psi_{0,xy}(1,0)=0$ and $\psi_{1,y}(1,0)=0$ 
   due to homogeneous Dirichlet boundary conditions of $\psi_1$, and
   \[
      v_2(1,y)+z_2(1,y)
         =v_2(1,y)
         =\left(\frac{\psi_{1,x}(0,y)}{b_1}-\frac{\alpha(y)}{b_1^3}-\frac{\alpha(y)}{b_1^2 \varepsilon}\right)E_1(1).
   \]
   We obtain for the corrections with index 3
   \[
      v_3(1,y)+z_3(1,y)
         =\widetilde{P}_2\left(x=\frac{1}{\varepsilon},y\right)E_1(1)+
          \bar{R}_1\left(x=\frac{1}{\varepsilon},\frac{y}{\varepsilon}\right)E_1(1)E_2(y).
   \]
   Finally,
   \[
      w_3(1,y)=\widetilde{Q}_2\left(1,\frac{y}{\varepsilon}\right)E_2(y)
   \]
   is bounded by a constant due to $P_2(t)\e^{-t}\leq C$ % 4\e^{-2}$ 
   for all $t$, where $P_2$ is a polynomial of order $2$. 
   The corrections with index 4 contribute similar bounds as an analogous argument shows. Combining the terms discussed so far, we obtain with some function $g$, bounded 
   independently of $\varepsilon$, and polynomials $P_2, P_3,$
   \[
      R(1,y)=\varepsilon g(y) E_1(1)+\left(\varepsilon^3 P_2\left(\frac{y}{\varepsilon}\right)+
             \varepsilon^4 P_3\left(\frac{y}{\varepsilon}\right)\right)E_2(y), \quad (0\leq y\leq 1).
   \]
   Thus, the $L^\infty$-norm over this part of the boundary is bounded by
   \[
      \|R(1,\cdot)\|_{L^\infty(0,1)}\leq C \varepsilon^3.
   \]
   Let us take a look at the outflow boundary at $x=0$. Here, we do not have an arbitrarily 
   small factor $E_1(1)$ in our estimates. We start with
   \begin{align*}
      R(0,y)
         &=\psi(0,y)-\psi_0(0,y)-\varepsilon \psi_1(0,y)-\varepsilon^2 \psi_2(0,y)\\
         &\quad-\varepsilon (v_1(0,y)+w_1(0,y))-\varepsilon^2(v_2(0,y)+w_2(0,y)+z_2(0,y))\\
         &\quad -\varepsilon^3(v_3(0,y)+w_3(0,y)+z_3(0,y))-\varepsilon^4(v_4(0,y)+w_4(0,y)+z_4(0,y)).
   \end{align*}
   The first two terms are zero due to the imposed Dirichlet conditions. 
   By definition of $\psi_1$ and $\psi_2$ we have
   \begin{align*}
      \psi_1(0,y)&=-v_1(0,y),\\
      \psi_2(0,y)&=-v_2(0,y)-w_2(0,y)-z_2(0,y).
   \end{align*}
   Employing the fact that $w_1(0,y)$ is zero since $\psi_0$ satisfies homogeneous Dirichlet conditions and, hence, $\psi_{0,y}(0,0)=0$, we get
   \begin{align*}
      R(0,y)
         &= -\varepsilon^3(v_3(0,y)+w_3(0,y)+z_3(0,y))-\varepsilon^4(v_4(0,y)+w_4(0,y)+z_4(0,y)).
   \end{align*}
  We obtain with the same methods as above 
   for some function $h$ bounded independently of $\varepsilon$ and polynomials $P_2, P_3$,
   \[
      R(0,y)=\varepsilon^3\left(h(y)+\left(P_2\left(\frac{y}{\varepsilon}\right)+ 
             \varepsilon P_3\left(\frac{y}{\varepsilon}\right)\right)E_2(y)\right)\quad(0\leq y\leq 1).
   \]
   Thus, the $L^\infty$-norm of this part of the boundary is also bounded by
   \[
      \|R(0,\cdot)\|_{L^\infty(0,1)}\leq C \varepsilon^3.
   \]
   On the remaining boundaries we use similar ideas and eventually obtain \eqref{eq:linftyR}.

   For the tangential derivative we estimate $R_y(0,y)$ and $R_y(1,y)$: We obtain
   \[
      \|R_y(0,\cdot)\|_{L^\infty(0,1)}\leq C \varepsilon^2
      \quad \mbox{and} \quad 
      \|R_y(1,\cdot)\|_{L^\infty(0,1)}\leq C \varepsilon^2
   \]
   due to $E_{2,y}(y)=-\frac{b_2}{\varepsilon}E_2(y)$. On the other boundaries we apply 
   the same techniques and conclude \eqref{eq:linftyRt}.
\end{proof}
\begin{lemma}\label{lem:Rnormrand}
   The normal derivative of $R$ on the boundary $\partial \Omega$ can be estimated by
   \begin{equation}\label{eq:normR}
      \|\partial_\boldn R\|_{L^\infty(\Gamma)}\leq C \varepsilon^3 \mbox{ and }
      \|\partial_\boldn R\|_{L^2(\Gamma)}\leq C \varepsilon^{7/2}.
   \end{equation}
\end{lemma}
\begin{proof}
   Again, we only exemplify the method of proof. The terms not discussed are estimated using the same techniques. 
   Let us start out with $R(1,y)$. Here, we have
   \begin{align*}
      R_x(1,y)
         &=\psi_x(1,y)-\psi_{0,x}(1,y)-\varepsilon \psi_{1,x}(1,y)-\varepsilon^2 \psi_{2,x}(1,y)\\
         &\quad-\varepsilon (v_{1,x}(1,y)+w_{1,x}(1,y))-\varepsilon^2(v_{2,x}(1,y)+w_{2,x}(1,y)+z_{2,x}(1,y))\\
         &\quad -\varepsilon^3(v_{3,x}(1,y)+w_{3,x}(1,y)+z_{3,x}(1,y))-\varepsilon^4(v_{4,x}(1,y)+w_{4,x}(1,y)+z_{4,x}(1,y)).
   \end{align*}
   Now, 
   \[
     \psi_x(1,y)-\psi_{0,x}(1,y)-\varepsilon \psi_{1,x}(1,y)-\varepsilon^2 \psi_{2,x}(1,y)=0
   \]
   due to the homogeneous Neumann boundary conditions at this boundary of the functions 
   $\psi,\psi_0,\psi_1$ and $\psi_2$. For the next terms we have
   \begin{align*}
      v_{1,x}(1,y)+w_{1,x}(1,y)
         &=-\frac{\psi_{0,x}(1,y)}{\varepsilon}E_1(1)+\frac{\psi_{0,xy}(1,y)}{b_2}E_2(y)\\
         &=-\frac{\psi_{0,x}(1,y)}{\varepsilon}E_1(1)
   \end{align*}
   because $\psi_0$ satisfies homogeneous Neumann boundary condition along $x=1$ and, hence, its $y$-derivative is zero. 
   The terms with index 2 are
   \begin{align*}
      w_{2,x}(1,y)
         &=\left(\frac{\psi_{1,xy}(1,0)}{b_2}-\frac{\beta_x(1)}{b_2^3}-\frac{\beta_x(1) y}{b_2^2 \varepsilon}\right)E_2(y)\\
         &=\frac{b_1 \psi_{0,xxy}(1,0)}{b_2^3}\left(1+\frac{b_2 y}{\varepsilon}\right)E_2(y)
          =0,
   \end{align*}
   due to the compatibility condition \eqref{eq:compcond3}, and
   \[
      v_{2,x}(1,y)+z_{2,x}(1,y)
         =-\frac{b_1}{\varepsilon}\left(\frac{\psi_{1,y}(1,0)}{b_1}-\frac{\alpha(y)}{b_1^3}-\frac{\alpha(y)}{b_1^2 \varepsilon}\right)E_1(1).
   \]
   The terms with index 3 contribute
   \begin{align*}
      v_{3,x}(1,y)+z_{3,x}(1,y)
         &=\left(\tilde{P}_{2,x}\left(\frac{1}{\varepsilon},y\right)-\frac{b_1}{\varepsilon}\tilde{P}_2\left(\frac{1}{\varepsilon},y\right)\right.\\
         &\quad\left.+\frac{w_2}{\varepsilon}E_2(y)-\frac{b_1}{\varepsilon}\left(\omega_1+\omega_2\frac{1}{\varepsilon}+\omega_3\frac{y}{\varepsilon}\right)E_2(y)\right)E_1(1)
   \end{align*}
   and
   \[
      w_{3,x}(1,y)=\tilde{Q}_{2,x}\left(1,\frac{y}{\varepsilon}\right)E_2(y)
   \]
   which is bounded by a constant due to $P_2(t)\e^{-t}\leq C$ for all $t$, where $P_2$ is a polynomial of order $2$. 
   The same arguments apply to the correction functions with index 4. 
   Combining all these terms, we have for some function $g$ bounded independently of 
   $\varepsilon$ and constants $c_7,c_8$ and $c_9$ independent of $\varepsilon$
   \[
      R_x(1,y)=g(y)E_1(1)+\varepsilon^3\left(c_7+c_8\left(\frac{y}{\varepsilon}\right)+c_9\left(\frac{y}{\varepsilon}\right)^2\right)E_2(y).
   \]
   Thus, we obtain
   \[
      \|R_x(1,\cdot)\|_{L^\infty(0,1)}\leq C \varepsilon^3
      \quad\mbox{and}\quad
      \|R_x(1,\cdot)\|_{L^2(0,1)}\leq C \varepsilon^{7/2}.
   \]
   On the outflow boundary at $x=0$ we have
   \begin{align*}
      R_x(0,y)
         &=\psi_x(0,y)-\psi_{0,x}(0,y)-\varepsilon \psi_{1,x}(0,y)-\varepsilon^2 \psi_{2,x}(0,y)\\
         &\quad-\varepsilon (v_{1,x}(0,y)+w_{1,x}(0,y))-\varepsilon^2(v_{2,x}(0,y)+w_{2,x}(0,y)+z_{2,x}(0,y)\\
         &\quad -\varepsilon^3(v_{3,x}(0,y)+w_{3,x}(0,y)+z_{3,x}(0,y))-\varepsilon^4(v_{4,x}(0,y)+w_{4,x}(0,y)+z_{4,x}(0,y)).
   \end{align*}
   The first term is zero and the next terms cancel with their corresponding layer-correction functions. So far, we have
   \[
      R_x(0,y)= -\varepsilon^3w_{3,x}(0,y)-\varepsilon^4(v_{4,x}(0,y)+w_{4,x}(0,y)+z_{4,x}(0,y)).
   \]
   The remaining terms are not vanishing and we have for some function $h$ bounded independently of $\varepsilon$ and constants $c_{10},c_{11}$ and $c_{12}$ independent of $\varepsilon$
   \[
      R_x(0,y)=\varepsilon^4 h(y)+\varepsilon^3\left(c_{10}+c_{11}\left(\frac{y}{\varepsilon}\right)+c_{12}\left(\frac{y}{\varepsilon}\right)^2\right)E_2(y).
   \]
   Thus, we obtain again
   \[
      \|R_x(0,\cdot)\|_{L^\infty(0,1)}\leq C \varepsilon^3 
      \quad\mbox{and}\quad
      \|R_x(0,\cdot)\|_{L^2(0,1)}\leq C \varepsilon^{7/2}.
   \]
   The normal derivative on the other sides can be estimated similarly and we obtain \eqref{eq:normR}.
\end{proof}
We summarize the results obtained so far in the next proposition. Recall that we assumed sufficiently smooth solutions to be existent.
\begin{prop}\label{prop:residual}
   Assume the compatibility conditions \eqref{eq:compcond1}, \eqref{eq:compcond2} and \eqref{eq:compcond3} 
   to be satisfied. We consider the approximation of $\psi$ given by 
   \[
      \Psi=\sum_{i=0}^2 \varepsilon^i \psi_i+\sum_{i=1}^4 \varepsilon^i v_i +\sum_{i=1}^4  \varepsilon^i w_i +\sum_{i=2}^4 \varepsilon^i z_i.
   \]
   Then, the residual $R=\psi-\Psi$ can be estimated by
   \[
      \|R\|_{L^\infty(\Omega)}\leq C \varepsilon^{3/2}
      \quad\mbox{and}\quad
      \|R\|_{H^1(\Omega)}\leq C \varepsilon^{3/2}.
   \]
\end{prop}
\begin{proof}
   We make use of Theorem \ref{th:stab2d}. For this, using Lemmas \ref{lem:LR}, \ref{lem:Rrand}, \ref{lem:Rnormrand}, we are left with an estimate of the $C^{1/2}$-norm on $\Gamma$ of $R$. But, with the help of the estimates in both the $C^1$- and $L^\infty$-norm it follows that
   \[
      \|R\|_{C^{1/2}(\Gamma)}
      \leq C \varepsilon^{5/2}.
   \]
   Hence, the assertion eventually follows from Theorem \ref{th:stab2d}.
\end{proof}
\begin{theorem}\label{th:solcom}
   Assume the compatibility conditions \eqref{eq:compcond1}, \eqref{eq:compcond2} and \eqref{eq:compcond3}
   to be satisfied. Then the solution $\psi$ of \eqref{eq:fourthorder} can be decomposed 
   in a regular part $S$ and two layer parts in the following form:
   \[
      \psi(x,y)=S(x,y)+\varepsilon H(x,y) E_1(x)+\varepsilon I(x,y) E_2(y), \quad (0\leq x,y\leq 1)
   \]
   where the functions $S,H$ and $I$ are independently bounded of $\varepsilon$.
\end{theorem}
\begin{proof}
   The result follows from the asymptotic expansion which was done in this section.
   Proposition \ref{prop:residual} yields that the residual $R$ is small enough and thus, 
   it can be incorporated into the smooth part $S$ due to the structure of the layer correcting 
   terms. The corner layer parts $z_3$ and $z_4$ are also incorporated into the smooth part.
\end{proof}
\begin{remark}
   The layers of $\psi$ are so-called weak layers. If we look at $\psi-S$, with $S$ being given by Theorem \ref{th:solcom}, we obtain
   \[
      \|\psi-S\|_{L^\infty(\Omega)}\leq C \varepsilon,
   \]
   but 
   \[
      \|\nabla(\psi-S)\|_{L^\infty(\Omega)}\leq C,
   \]
   that is the layers are visible in the first derivative for $\eps\to 0$.
\end{remark}

\subsection{Asymptotic expansion without compatibility conditions}
With the rationale presented, we have seen that an asymptotic expansion of arbitrary order is possible upon imposing certain 
compatibility conditions. In fact, for the construction of $\psi_2$ we needed
\eqref{eq:compcond2} and \eqref{eq:compcond3} to be satisfied. 
We make now an asymptotic expansion without any compatibility conditions and demonstrate that 
we will lose an $\varepsilon$-order in the estimates of the residual $R$, while we keep in mind that we assume the solutions occuring to be sufficiently smooth.

We keep the construction of the functions $\psi_0,\psi_1, v_1,w_1,v_2,w_2$ and $z_2$. 
In the next step, we formally set $\psi_2=0$. This implies that we impose homogeneous 
boundary conditions for $v_3$ and $w_3$. For $z_3$ we choose an arbitrary, exponentially 
decaying function that fulfils a corner-correction type problem, given by
\begin{align*}
   \bar{\Delta}^2 \bar{z}_3 +(b\cdot \bar{\nabla})\bar{\Delta} \bar{z}_3
      &=c\bar{\Delta} \bar{z}_2 \notag \\
      &=-\frac{c \psi_{0,xy}(0,0)\exp\{-b_1\xi\}\exp\{-b_2\eta\}(b_1^2+b_2^2)}{b_1b_2}\,\hspace*{2em}\mbox{in }(0,\infty)\times(0,\infty).
\end{align*}
Its boundary data only have to be bounded by a constant and need not to correct other data, 
thus we have some freedom in choosing $\bar{z}_3$. We take
\[
   z_3(x,y)=\frac{c \psi_{0,xy}(0,0)}{b_1b_2(b_1+b_2)}\left(\frac{x}{\varepsilon}+\frac{y}{\varepsilon}\right)E_1(x)E_2(y).
\]
Next, we consider an approximation $\Psi^\textnormal{new}$ of $\psi$ given by
\[
   \Psi^\textnormal{new}=\psi_0+\varepsilon\psi_1+\varepsilon v_1+\varepsilon^2  v_2 +\varepsilon^3 v_3+
                   \varepsilon w_1+\varepsilon^2 w_2+\varepsilon^3 w_3+\varepsilon^2 z_2 +\varepsilon^3 z_3.
\]
Following the strategies exemplified in the previous section, we derive estimates for the new residual.
\begin{lemma}
   For the residual $R^\textnormal{new}=\psi-\Psi^\textnormal{new}$ we get the following estimates:
   \begin{align*}
      \|LR^\textnormal{new}\|_{L^2(\Omega)}&\leq C \varepsilon^{3/2},\\
      \|R^\textnormal{new}\|_{L^\infty(\Gamma)}&\leq C \varepsilon^2,&
%       \|R^\textnormal{new}\|_{L^2(\Gamma)}&\leq C \varepsilon^2,\\
      \|\partial_t R^\textnormal{new}\|_{L^\infty(\Gamma)}&\leq C \varepsilon,\\
      \|\partial_\boldn R^\textnormal{new}\|_{L^\infty(\Gamma)}&\leq C \varepsilon^2,&
      \|\partial_\boldn R^\textnormal{new}\|_{L^2(\Gamma)}&\leq C \varepsilon^{5/2}\\
      \intertext{and finally}
      \|R^\textnormal{new}\|_{L^\infty(\Omega)}&\leq C \varepsilon^{1/2}.
   \end{align*}
\end{lemma}
\begin{proof}
   The upper bounds follow by a straightforward calculation in the same way as in Lemmas \ref{lem:LR}, 
   \ref{lem:Rrand} and \ref{lem:Rnormrand}. For obtaining the last inequality, one has to use Theorem~\ref{th:stab2d} in addition.
\end{proof}
So we see that the $\varepsilon$-order of the residual $R^\textnormal{new}$ is not small enough to have a solution 
decomposition like in Theorem \ref{th:solcom}. We can only show the following.
\begin{prop}\label{prop:solcomnew}
   Consider equation \eqref{eq:fourthorder}. The solution $\psi$ can be represented as
   \[
      \psi(x,y)=S(x,y)+\varepsilon H(x,y) E_1(x)+\varepsilon I(x,y) E_2(y)+\varepsilon^2 J(x,y) E_1(x) E_2(y) +R^\textnormal{new}(x,y),\quad(0\leq x,y\leq 1)
   \]
   where the functions $S,H,I,J$ are bounded independently of $\varepsilon$ and the residual is of order $\varepsilon^{1/2}$.
\end{prop}

\subsection*{Acknowledgement}
We would like to thank J\"urgen Rossmann for his help with providing a stability result 
for the solution of singularly perturbed fourth-order problems, see Theorem~\ref{th:stab2d}. 

\noindent This work was supported by the German Research Foundation (DFG) under Project No. FR 3052/2-1.
\bibliography{asym_arxiv}

\def\cprime{$'$} \def\cprime{$'$} \def\cprime{$'$}
\begin{thebibliography}{10}

\bibitem{Adams03}
R.~A. Adams and J.~J.~F. Fournier.
\newblock {\em Sobolev spaces}, volume 140 of {\em Pure and Applied Mathematics
  (Amsterdam)}.
\newblock Elsevier/Academic Press, Amsterdam, second edition, 2003.

\bibitem{BR80}
H.~Blum and R.~Rannacher.
\newblock On the boundary value problem of the biharmonic operator on domains
  with angular corners.
\newblock {\em Math. Methods Appl. Sci.}, 2(4):556--581, 1980.

\bibitem{BrennerSung05}
S.~C. Brenner and L.-Y. Sung.
\newblock {$C^0$} interior penalty methods for fourth order elliptic boundary
  value problems on polygonal domains.
\newblock {\em J. Sci. Comput.}, 22-23(1-3):83--118, 2005.

\bibitem{JF96}
E.~M. de~Jager and J.~Furu.
\newblock {\em The theory of singular perturbations}, volume~42 of {\em
  North-Holland Series in Applied Mathematics and Mechanics}.
\newblock North-Holland Publishing Co., Amsterdam, 1996.

\bibitem{EngelGHLMT02}
G.~Engel, K.~Garikipati, T.J.R. Hughes, M.G. Larson, L.~Mazzei, and R.L.
  Taylor.
\newblock Continuous/discontinuous finite element approximations of
  fourth-order elliptic problems in structural and continuum mechanics with
  applications to thin beams and plates, and strain gradient elasticity.
\newblock {\em Computer Methods in Applied Mechanics and Engineering},
  191(34):3669--3750, 2002.

\bibitem{EG04}
A.~Ern and J.-L. Guermond.
\newblock {\em Theory and practice of finite elements}, volume 159 of {\em
  Applied Mathematical Sciences}.
\newblock Springer-Verlag, New York, 2004.

\bibitem{FrRW12}
S.~Franz, H.-G. Roos, and A.~Wachtel.
\newblock A {$C^0$} interior penalty method for a singularly-perturbed
  fourth-order elliptic problem on a layer-adapted mesh.
\newblock {\em Numer. Methods Partial Differential Equations}, 30(3):838--861,
  2014.

\bibitem{GuzmanLeykekhmanNeilan12}
J.~Guzm\'{a}n, D.~Leykekhman, and M.~Neilan.
\newblock A family of non-conforming elements and the analysis of nitsche’s
  method for a singularly perturbed fourth order problem.
\newblock {\em Calcolo}, 49(2):95--125, 2012.

\bibitem{kato}
T.~Kato.
\newblock {\em Perturbation theory for linear operators}.
\newblock Springer-Verlag, Berlin-New York, second edition, 1976.
\newblock Grundlehren der Mathematischen Wissenschaften, Band 132.

\bibitem{KSt05}
R.~B. Kellogg and M.~Stynes.
\newblock {Sharpened and corrected version of: Corner singularities and
  boundary layers in a simple convection-diffusion problem}.
\newblock {\em J. Differential Equations}, 213(1):81--120, 2005.

\bibitem{KSt07}
R.~B. Kellogg and M.~Stynes.
\newblock {Sharpened bounds for corner singularities and boundary layers in a
  simple convection-diffusion problem}.
\newblock {\em Appl. Math. Lett.}, 20(5):539--544, 2007.

\bibitem{KMR97}
V.~A. Kozlov, V.~G. Maz{\cprime}ya, and J.~Rossmann.
\newblock {\em Elliptic boundary value problems in domains with point
  singularities}, volume~52 of {\em Mathematical Surveys and Monographs}.
\newblock American Mathematical Society, Providence, RI, 1997.

\bibitem{LS01}
T.~Lin{\ss} and M.~Stynes.
\newblock Asymptotic analysis and {S}hishkin-type decomposition for an elliptic
  convection-diffusion problem.
\newblock {\em J. Math. Anal. Appl.}, 261(2):604--632, 2001.

\bibitem{MR91}
V.~G. Maz{\cprime}ya and J.~Rossmann.
\newblock On the {A}gmon-{M}iranda maximum principle for solutions of elliptic
  equations in polyhedral and polygonal domains.
\newblock {\em Ann. Global Anal. Geom.}, 9(3):253--303, 1991.

\bibitem{necas}
J.~Ne{\v{c}}as.
\newblock {\em Direct methods in the theory of elliptic equations}.
\newblock Springer Monographs in Mathematics. Springer, Heidelberg, 2012.
\newblock Translated from the 1967 French original by Gerard Tronel and Alois
  Kufner, Editorial coordination and preface by {\v{S}}{\'a}rka
  Ne{\v{c}}asov{\'a} and a contribution by Christian G. Simader.

\bibitem{RPersonal}
J.~Rossmann.
\newblock Personal Communication, June 2015.

\bibitem{MR1323738}
G.~F. Sun and M.~Stynes.
\newblock Finite-element methods for singularly perturbed high-order elliptic
  two-point boundary value problems. {II}. {C}onvection-diffusion-type
  problems.
\newblock {\em IMA J. Numer. Anal.}, 15(2):197--219, 1995.

\bibitem{Tr95}
H.~Triebel.
\newblock {\em {Interpolation theory, function spaces, differential operators.
  2nd rev. a. enl. ed.}}
\newblock Leipzig: Barth, 2nd rev. a. enl. ed. edition, 1995.

\bibitem{TW14}
S.~Trostorff and M.~Waurick.
\newblock {A note on elliptic type boundary value problems with maximal
  monotone relations.}
\newblock {\em {Math. Nachr.}}, 287(13):1545--1558, 2014.

\bibitem{MR2548434}
O.~S. Zikirov.
\newblock A non-local boundary value problem for third-order linear partial
  differential equation of composite type.
\newblock {\em Math. Model. Anal.}, 14(3):407--421, 2009.

\end{thebibliography}
\bibliographystyle{plain}
\end{document}